\documentclass[11pt]{article}

\usepackage[margin=1in]{geometry}
\usepackage{amsmath,amssymb,amsthm}
\usepackage{mathtools}
\usepackage{booktabs}
\usepackage{enumitem}
\usepackage{tikz}
\usepackage{pgfplots}
\usepackage[hidelinks,pdfencoding=auto,psdextra]{hyperref}
\usepackage[round,authoryear]{natbib}
\pgfplotsset{compat=1.17}

\usepackage{graphicx}
\usepackage{array}

\usepackage{bbm}

\usepackage{threeparttable}

\newcommand{\NN}{\ensuremath{\mathbb{N}}}

\newcommand{\RR}{\ensuremath{\mathbbm{R}}}

\newcommand{\norm}[1]{\lVert #1 \rVert}
\newcommand{\scal}[2]{\langle #1, #2 \rangle}

\newcommand{\pCRM}{\ensuremath{\mathrm{pCRM}}}
\DeclareMathOperator{\dist}{dist}
\DeclareMathOperator{\Fix}{Fix}
\DeclareMathOperator{\aff}{aff}
\DeclareMathOperator{\circum}{circumcenter}
\DeclareMathOperator{\Span}{span}

\DeclareMathOperator{\Id}{Id}

\newtheorem{theorem}{Theorem}[section]
\newtheorem{proposition}[theorem]{Proposition}
\newtheorem{lemma}[theorem]{Lemma}
\newtheorem{corollary}[theorem]{Corollary}
\newtheorem{definition}[theorem]{Definition}
\newtheorem{fact}[theorem]{Fact}
\theoremstyle{definition}
\newtheorem{example}[theorem]{Example}
\theoremstyle{remark}
\newtheorem{remark}[theorem]{Remark}

\title{On the sharp linear convergence rate of the circumcentered--reflection method on subspaces}
\author{Yunier Bello-Cruz\thanks{School of Mathematical and
   Statistical Sciences, Northern Illinois University, DeKalb, IL 60115, USA.
   E-mail: \texttt{yunierbello@niu.edu}.
   ORCID: \url{https://orcid.org/0000-0002-7877-5688}.}}
\date{\today}

\begin{document}
\maketitle
\begin{abstract}
\noindent
For two subspaces $U,V\subseteq\RR^n$, the circumcentered--reflection
method (CRM) of Behling, Bello-Cruz, and Santos~\citeyearpar{BBS2018}
computes the projection onto $U\cap V$ using only the reflections
across $U$ and $V$, with known linear-convergence rate equal to
the cosine of the Friedrichs angle. We prove that, when CRM is
initialized in $V$, it contracts at the strictly smaller rate
$\rho_V=(\sin^2\theta_p-\sin^2\theta_F)/(\sin^2\theta_p+\sin^2\theta_F)$,
where $\theta_F\in(0,\pi/2]$ is the Friedrichs angle, whose
cosine we denote by $c_F\coloneqq\cos\theta_F$, and
$\theta_p\in[\theta_F,\pi/2]$ is the largest principal angle
between $U$ and $V$. The bound is sharp, attained on an explicit ray in
$V$, and optimal among parameter-free single-step iterations. The
constant itself is not new: Bauschke, Bello-Cruz, Nghia, Phan, and
Wang~\citeyearpar{BBNPW2016} identified it as the optimal rate of the
relaxed alternating-projection family and of their adaptive
linesearch map $B_T$. Our contribution is that the parameter-free
geometric circumcenter attains it as well, via Kantorovich's
inequality applied to a single self-adjoint operator on $V$.
Restricted to $V$, CRM coincides pointwise with the linesearch
maps $A_T$ and $B_T$ from the Gubin--Polyak--Raik
framework~\citep{GPR1967}. We further prove $\rho_V<c_F^2$ whenever
$\theta_F<\pi/2$, with one-step convergence exactly when
$\theta_F=\theta_p$. Over-reflecting either or both of $R_U$, $R_V$ inside
the circumcenter does not help. Going faster than $\rho_V$
universally requires memory: Chebyshev semi-iteration applied to
$P_VP_U$ attains a strictly smaller rate, beating $\rho_V$ by a
factor at most $2$, attained in the limit $\theta_F\to\theta_p$.
\end{abstract}


\section{Introduction}\label{sec:intro}

\subsection{The problem and the known rates}\label{ssec:intro-problem}

The computation of the best approximation
$P_{U\cap V}(x_0)$ to a point $x_0$ from the intersection
$U\cap V$ of two subspaces of $\RR^n$, using only the individual
projections $P_U$ and $P_V$, is a model problem of numerical
analysis. Iterative projection methods solve it; the rate at
which those iterations converge is the subject of this paper. The
same computation underlies numerical linear algebra,
approximation theory, statistics, and signal processing, with
applications to image and signal
reconstruction~\citep{Combettes1996}, compressed
sensing~\citep{ESQB2007}, phase retrieval, matrix completion, and
convex feasibility. The two-subspace case is the cleanest setting
for studying convergence rates of iterative best-approximation
algorithms: the geometry is governed entirely by the principal
angles between $U$ and $V$, and explicit sharp bounds are
accessible. The principal-angle framework, going back to Jordan
and developed for projection methods by Aronszajn, von Neumann,
and Friedrichs, is central to the approximation theory of
orthogonal projections; we refer to the monographs of
Deutsch~\citep{Deutsch} and Escalante and
Raydan~\citep{EscalanteRaydan2011} for systematic treatments.

For $x_0\in\RR^n$, write
\begin{equation}\label{eq:best-approximation}
   \bar x_0  \coloneqq  P_{U\cap V}(x_0)
\end{equation}
for the best-approximation point. A range of iterative schemes
solve~\eqref{eq:best-approximation} using only the projections
$P_U,P_V$ and the corresponding reflections $R_U=2P_U-\Id$,
$R_V=2P_V-\Id$. Five of them are relevant here. The method of
alternating projections (MAP) iterates the linear operator
$T\coloneqq P_VP_U$. The Douglas--Rachford method (DRM) iterates
the average $T_{\mathrm{DR}}\coloneqq\tfrac12(\Id+R_VR_U)$, whose
linear-convergence rate on two subspaces was identified as
$c_F=\cos\theta_F$ by Bauschke, Bello-Cruz, Nghia, Phan, and
Wang~\citep{BCNPW2014}. Both methods admit one-parameter
relaxations: the relaxed alternating-projection family
\begin{equation}\label{eq:S-mu-intro}
   S_\mu(x)  \coloneqq  x+\mu\bigl(T(x)-x\bigr)
     =  (1-\mu)x+\mu T(x),\quad\mu\in\RR,
\end{equation}
recovers MAP at $\mu=1$, and the generalized Douglas--Rachford
family
\begin{equation}\label{eq:gDR-intro}
   T_{\mathrm{DR},\kappa}(x)  \coloneqq
     x+\kappa\bigl(T_{\mathrm{DR}}(x)-x\bigr)
     =  (1-\kappa)x+\kappa\,T_{\mathrm{DR}}(x),\quad\kappa\in\RR,
\end{equation}
recovers DRM at $\kappa=1$. The optimal-rate analysis of both
relaxations on two subspaces is the subject of~\citep{BBNPW2016}.
Finally, the circumcentered--reflection method introduced
in~\citep{BBS2018} iterates the nonlinear operator
\begin{equation}\label{eq:CT-intro}
   C_T(x)\coloneqq\circum\{x,\,R_Ux,\,R_VR_Ux\},
\end{equation}
defined as the unique point of $W_x\coloneqq\aff\{x,R_Ux,R_VR_Ux\}$
equidistant from the three vertices when these are non-collinear,
and extended to degenerate configurations by the
projection-of-equidistant-point convention
\eqref{eq:circum-convention} below.

Convergence rates of the four linear methods are well understood
in terms of the principal angles between $U$ and $V$. Set
$p\coloneqq\dim V$ and $s\coloneqq\dim(U\cap V)$, so that $0\le s<p$
by our standing assumption $V\not\subset U$. The principal angles
are $0\le\theta_1\le\cdots\le\theta_p\le\pi/2$, the first $s$ of
them equal to zero. The Friedrichs angle is
$\theta_F\coloneqq\theta_{s+1}\in(0,\pi/2]$, the smallest positive
principal angle, and we write $c_F\coloneqq\cos\theta_F$.
Aronszajn~\citep{Aronszajn1950} established convergence of MAP to
$P_{U\cap V}$ in the two-subspace setting; the sharp rate
$\norm{T^k-P_{U\cap V}}=c_F^{2k-1}$ was obtained by Kayalar and
Weinert~\citep{Kayalar1988}, so MAP contracts at asymptotic rate
$c_F^2$. See also \citep[Theorem~9.31]{Deutsch}. The
Douglas--Rachford rate is
$\norm{T_{\mathrm{DR}}^k-P_{\Fix T_{\mathrm{DR}}}}=c_F$
\citep[Theorem~4.1]{BCNPW2014}, and the original CRM analysis
\citep{BBS2018} obtains the same rate $c_F$.

For the relaxed family $\{S_\mu\}$, Bauschke, Bello-Cruz, Nghia,
Phan, and Wang~\citep[Theorem~3.7]{BBNPW2016} identified the
optimal parameter and rate. Writing $r(\mu)$ for the operator
norm of $S_\mu-P_{U\cap V}$ restricted to $V\cap(U\cap V)^\perp$,
the unique minimizer is
\begin{equation}\label{eq:mu-star-intro}
   \mu^\star=\frac{2}{\sin^2\theta_F+\sin^2\theta_p},
   \quad
   r(\mu^\star)=\rho_V\coloneqq
   \frac{\sin^2\theta_p-\sin^2\theta_F}{\sin^2\theta_p+\sin^2\theta_F}.
\end{equation}
A short calculation gives the hierarchy
\begin{equation}\label{eq:rate-hierarchy}
   \rho_V\le\frac{c_F^2}{2-c_F^2}\le c_F^2\le c_F<1,
\end{equation}
proved in Propositions~\ref{prop:strict}
and~\ref{prop:worst-case}, with the first inequality strict iff
$\theta_p<\pi/2$ and the next two strict iff $\theta_F<\pi/2$.
In particular, $\rho_V$ is strictly smaller than the rates of
MAP, DRM, and the original CRM analysis~\citep{BBS2018} whenever
$\theta_F<\pi/2$ and $\theta_p<\pi/2$.

Two further families admit smaller worst-case rates than $\rho_V$
when their parameters are tuned to the Friedrichs angle. The
generalized alternating-projections iteration of F\"alt and
Giselsson~\citep{FaltGiselsson2017} and the averaged alternating
modified reflections method of Arag\'on Artacho and
Campoy~\citep{AragonCampoy2018,AragonCampoy2019} both attain the
worst-case rate
$(1-\sin\theta_F)/(1+\sin\theta_F)$ at parameters fixed to
$\theta_F$. At $\theta_p=\pi/2$, this is strictly smaller than
$\rho_V=(1-\sin^2\theta_F)/(1+\sin^2\theta_F)$. The trade-off is
parametric: both methods require advance knowledge of the
Friedrichs angle to set the parameters, while $C_T$ does not; both rates also depend on
$\theta_F$ alone, while $\rho_V$ depends on both principal
angles and vanishes when $\theta_F=\theta_p$. The detailed
comparison is in Table~\ref{tab:method-comparison} and
Section~\ref{ssec:aamr-comparison}.

One caution when reading these constants side by side: for MAP
the value $c_F^2$ is the asymptotic ratio of the operator-power
norms in Fact~\ref{fact:friedrichs-norms},
$\norm{(P_VP_U)^k-P_{U\cap V}}=c_F^{2k-1}$, not a per-step
guarantee, whereas $\rho_V$ bounds every individual step of
$C_T$ on $V$ by Theorem~\ref{thm:main}.

\subsection{Main result: from the circumcenter to a sharp rate}\label{ssec:intro-main}

We prove a sharp linear rate for the circumcentered--reflection
method started in $V$. The argument has three links, and the
paper is organized around them.

The first link is geometric. Every point of $U\cap V$ is
equidistant from the three CRM vertices, so the circumcenter is
the orthogonal projection of $\bar x=P_{U\cap V}(x)$ onto the
plane $W_x$ of the vertices (Lemma~\ref{lem:BBS},
from~\citep{BBS2018}). The second link turns this projection
into a formula. For $v\in V$ the plane $W_v$ is spanned, around
$v$, by the orthogonal vectors $v-T(v)$ and $P_{V^\perp}P_Uv$,
and projecting $\bar v=P_{U\cap V}(v)$ onto it gives
\begin{equation}\label{eq:CT-on-line-intro}
   C_T(v)=v+\mu_v\bigl(T(v)-v\bigr),
   \qquad
   \mu_v=\frac{\dist^2(v,U)}{\norm{T(v)-v}^2}
\end{equation}
(Theorem~\ref{thm:linesearch}). Started in $V$, CRM is therefore
a relaxed alternating-projection step whose relaxation is read
off from the geometry. The step $\mu_v$ minimizes the distance
to $\bar v$ along the line $\{v+\mu(T(v)-v):\mu\in\RR\}$
(Corollary~\ref{lem:mu-optimal}), the line traced by the relaxed
family $\{S_\mu\}$ of~\eqref{eq:S-mu-intro}; the optimal
constant $\mu^\star$ of~\eqref{eq:mu-star-intro} depends on the
principal angles, and $\mu_v$ does not. The third link is a
spectral estimate. With $w=v-\bar v$ and $M=(\Id-P_VP_U)|_V$,
formula~\eqref{eq:CT-on-line-intro} gives
\[
   \frac{\norm{C_T(v)-\bar v}^2}{\norm{v-\bar v}^2}
   =1-\frac{\scal{w}{Mw}^2}{\norm{w}^2\,\scal{w}{M^2w}},
\]
and $M$ is self-adjoint with spectrum in
$\{0\}\cup[\sin^2\theta_F,\sin^2\theta_p]$
(Lemma~\ref{lem:spectrum}). Kantorovich's
inequality~\citep{Kantorovich1948} bounds the last fraction from
below by $1-\rho_V^2$, and its extremizers are the worst-case
points (Theorem~\ref{thm:rate}). We are not aware of an earlier
use of Kantorovich's inequality in the rate analysis of
projection methods. The degenerate configurations need no
separate argument: $T(v)=v$ forces $v\in U\cap V$, and $C_T$
commutes with translations by elements of $U\cap V$.

The three links give the principal result of the paper.

\begin{theorem}[Optimal linear rate of CRM on $V$]\label{thm:main}
Let $T=P_VP_U$. For every $v\in V$,
\begin{equation}\label{eq:main-rate}
   \norm{C_T(v)-P_{U\cap V}(v)}  \le  \rho_V\,\norm{v-P_{U\cap V}(v)},
\end{equation}
with $\rho_V$ as in~\eqref{eq:mu-star-intro}, and the bound is
sharp: equality holds in~\eqref{eq:main-rate} at every point of
the set
\begin{equation}\label{eq:main-rate-sharp}
   U\cap V  +  \RR\bigl(\sin\theta_p\,f_{s+1}+\sin\theta_F\,f_p\bigr),
\end{equation}
where $\{f_k\}$ is the principal-angle frame of $V$ from
Theorem~\ref{lem:frame}. Consequently, for every $x_0\in\RR^n$,
the sequence $v_0\coloneqq P_V(x_0)$,
$v_{k+1}\coloneqq C_T(v_k)$, $k\in\NN$, satisfies $v_k\in V$ and
$P_{U\cap V}(v_k)=\bar x_0$ for every $k\in\NN$, the shadow on
$U\cap V$ being invariant along the orbit, whence
\begin{equation}\label{eq:main-rate-orbit}
   \norm{v_k-\bar x_0}  \le  \rho_V^{\,k}\,\norm{v_0-\bar x_0}
   \quad(k\in\NN),
\end{equation}
and $(v_k)_{k\geq 0}$ converges linearly with rate $\rho_V$ to
the best approximation $\bar x_0=P_{U\cap V}(x_0)$. In
particular, $\rho_V\le c_F^2$, with strict inequality whenever
$\theta_F<\pi/2$.
\end{theorem}

Theorem~\ref{thm:main} sharpens the rate $c_F$
of~\citep{BBS2018} by a strict factor whenever $\theta_F<\pi/2$,
and the new rate is parameter-free: $C_T$ attains it using only
$P_U$ and $P_V$. Because $C_T$ is nonlinear, the operator-norm
notion of a sharp rate does not apply directly; sharpness here
means that the per-step bound~\eqref{eq:main-rate} is attained.
Figure~\ref{fig:roadmap} displays the three links and places
$\rho_V$ among the known rates.

\begin{figure}[htbp]
\centering
\small
\begin{gather*}
C_T(v)
   \;\overset{\text{Lem.~\ref{lem:BBS}}}{=}\;
   P_{W_v}(\bar v)
   \;\overset{\text{Thm.~\ref{thm:linesearch}}}{=}\;
   v+\mu_v\bigl(T(v)-v\bigr)
   \;\overset{\text{Thm.~\ref{thm:three-agree-on-V}}}{=}\;
   A_T(v)\;=\;B_T(v)
   \qquad(v\in V)
\\[4pt]
\mu_v=\frac{\dist^2(v,U)}{\norm{T(v)-v}^2}
   =\operatorname*{arg\,min}_{\mu\in\RR}\norm{S_\mu(v)-\bar v}
   \ \ \text{(Cor.~\ref{lem:mu-optimal})}
\;\;\overset{\text{Thm.~\ref{thm:rate}}}{\Longrightarrow}\;\;
\norm{C_T(v)-\bar v}\le\rho_V\,\norm{v-\bar v}
\\[6pt]
\underset{\substack{\text{two-step memory}\\
   \text{(Thm.~\ref{thm:cheb-strict})}}}
   {\frac{\sin\theta_p-\sin\theta_F}{\sin\theta_p+\sin\theta_F}}
\;\;<\;\;
\underset{\substack{\text{$C_T$ on $V$, sharp at $v^\star$ (Thm.~\ref{thm:main})}\\
   \text{optimal single-step (Cor.~\ref{prop:single-step-opt})}}}
   {\rho_V=\frac{\sin^2\theta_p-\sin^2\theta_F}{\sin^2\theta_p+\sin^2\theta_F}}
\;\;\overset{\text{Prop.~\ref{prop:worst-case}}}{\le}\;\;
\frac{c_F^2}{2-c_F^2}
\;\;\overset{\text{Prop.~\ref{prop:strict}}}{<}\;\;
\underset{\text{MAP}}{c_F^2\vphantom{\frac{c_F^2}{2}}}
\;\;<\;\;
\underset{\text{DRM, CRM on $\RR^n$}}{c_F\vphantom{\frac{c_F^2}{2}}}
\\[5pt]
\rho_V=0\iff\theta_F=\theta_p
   \;\;\text{(Cor.~\ref{cor:one-step})},
\qquad
\norm{C_T^{(\lambda,\eta)}(v)-\bar v}\ge\norm{C_T(v)-\bar v}
   \ \text{on}\ V
   \;\;\text{(Prop.~\ref{prop:over-reflection-double})}
\end{gather*}
\caption{The main line of the argument and the rate landscape of
Theorem~\ref{thm:main}, with each relation labeled by the
statement that proves it. First row: on $V$, the circumcenter is
the projection of $\bar v=P_{U\cap V}(v)$ onto the plane $W_v$
of the three vertices; this projection is a relaxed
alternating-projection step, and it coincides with the two
Gubin--Polyak--Raik-type maps
of~\eqref{eq:AT-our-order}--\eqref{eq:BT-our-order}. Second row:
the step $\mu_v$ is the error-minimizing relaxation along
$T(v)-v$, and Kantorovich's inequality
(Lemma~\ref{lem:kantorovich}) turns this into the sharp rate.
Third row: the rate hierarchy for $\theta_F\in(0,\pi/2)$ and
$\theta_p\in(\theta_F,\pi/2]$; the inequality
$\rho_V\le c_F^2/(2-c_F^2)$ is an equality exactly at
$\theta_p=\pi/2$, and the leftmost (Chebyshev) rate requires
two-step memory and knowledge of both angles. Fourth row: the
degenerate end of the scale, and the relaxed-reflection variant
$C_T^{(\lambda,\eta)}$, in which $R_U,R_V$ are replaced by
$R_U^{(\lambda)},R_V^{(\eta)}$ inside the circumcenter.}
\label{fig:roadmap}
\end{figure}

\subsection{Relation to prior work}\label{ssec:intro-related}

Adaptive stepsize acceleration of alternating projections has a
long history, originating with Gubin, Polyak, and
Raik~\citep{GPR1967}, with refinements
in~\citep{GearhartKoshy1989,BDHP2003,BCK2006}. Bauschke et
al.~\citep{BBNPW2016} bring it to bear on the rate-$\rho_V$
analysis, introducing two adaptive maps. In our order convention
(project onto $U$ first, then onto $V$), these read
\begin{align}
   A_T(x)&=(1-\lambda_x)\,x+\lambda_x\,P_VP_Ux,
   &\lambda_x&=\frac{\scal{x-P_VP_Ux}{x}}{\norm{x-P_VP_Ux}^2},\label{eq:AT-our-order}\\
   B_T(x)&=(1-\mu_x^{\mathrm B})\,P_Vx+\mu_x^{\mathrm B}\,P_VP_Ux,
   &\mu_x^{\mathrm B}&=\frac{\scal{P_Vx-P_VP_Ux}{x}}{\norm{P_Vx-P_VP_Ux}^2},\label{eq:BT-our-order}
\end{align}
with the convention $\lambda_x=\mu_x^{\mathrm B}=1$ when the
respective denominators vanish. Bauschke et al.\ define $A_T$ and
$B_T$ in the opposite order and show that the two maps coincide
on the ``project last'' subspace; restated in our order, they
coincide on $V$. The map $B_T$ achieves the rate $\rho_V$ on
$\RR^n$ after one warm-start projection
\citep[Theorem~4.5]{BBNPW2016}.

On $V$ the three maps coincide:
\begin{equation}\label{eq:three-agree-on-V}
   A_T(v)  =  B_T(v)  =  C_T(v)
   \quad(v\in V),
\end{equation}
proved as Theorem~\ref{thm:three-agree-on-V}. The maps $A_T$ and
$B_T$ are linesearches by definition and $C_T$ is a circumcenter
by definition; restricted to $V$ they produce the same point on
the same line at the same step, while off $V$ they differ
(Section~\ref{ssec:off-V}). The constant $\rho_V$ itself was
identified in~\citep{BBNPW2016} as the optimal rate of
$\{S_\mu\}$ and of $B_T$, and~\eqref{eq:three-agree-on-V}
combined with~\citep[Theorem~4.5]{BBNPW2016} would give the
inequality~\eqref{eq:main-rate}. That route yields neither the
exact contraction factor with its maximizers
(Corollary~\ref{prop:rho-sharp}), nor the optimality of $\rho_V$
among all steps along $T(v)-v$
(Corollary~\ref{prop:single-step-opt}), nor the strict
inequality $\rho_V<c_F^2$ for $\theta_F<\pi/2$
(Proposition~\ref{prop:strict}), nor the worst-case envelope
$\rho_V\le c_F^2/(2-c_F^2)$ attained at $\theta_p=\pi/2$
(Proposition~\ref{prop:worst-case}). These need the spectral
analysis of Section~\ref{sec:rate}, which the matrix-power
asymptotics of~\citep[Section~2]{BBNPW2016} do not produce.

The representation~\eqref{eq:CT-on-line-intro} also has a
precursor in the circumcenter literature. For a closed convex
set $A$ and an affine subspace $B$, the CRM point computed from
$x\in B$ is the projection of $x$ onto $H(x)\cap B$ for a
halfspace $H(x)\supseteq A$
\citep[Proposition~2.3(ii)]{Behling-SIOPT2024}; see
also~\citep{BBS2021}. Working in Pierra's product space,
Behling, Bello-Cruz, Iusem, Liu, and
Santos~\citep[Lemma~4.2 and~(3.5)--(3.6)]{Behling-SIOPT2024}
write the resulting step as an extrapolated simultaneous
projection whose coefficient is the ratio of the mean squared
projection residual to the squared norm of the averaged
residual; there the first set changes from one iteration to the
next and need not be a subspace. Our $\mu_v$ is the same ratio
for the pair $(U,V)$, with the average over the product space
replaced by the projection onto $V$. Three facts are specific to
two subspaces, and the rest of the paper rests on them: $\mu_v$
is the error-minimizing step along $T(v)-v$
(Corollary~\ref{lem:mu-optimal}); it coincides with the steps of
$A_T$ and $B_T$, which ties the circumcenter literature to the
linesearch lineage above (Theorem~\ref{thm:three-agree-on-V});
and the direction $T(v)-v=-M(v-\bar v)$ is a fixed self-adjoint
operator applied to the error, so that $\mu_v$ is a
Rayleigh-type quotient of $M$ and Kantorovich's inequality
applies. Remark~\ref{rem:known-formulas} spells out the
correspondence.

The CRM was introduced in~\citep{BBS2018} for two subspaces, with
linear rate $c_F$ established in~\citep{BBS2018-ORL}, and was
later extended to block-wise variants~\citep{BBS2020}, to the
convex feasibility problem with multiple sets~\citep{BBS2021} and
with approximate reflections~\citep{AABBIS2022}, and to nonsubspace
settings via centralization~\citep{Behling-MP2024,Behling-COA2024};
with perturbed approximate projections, a circumcenter method
for the convex feasibility problem terminates finitely under a
Slater condition~\citep{Behling-SIOPT2024}. CRM has
also been shown to achieve strictly better rates than
alternating projections~\citep{ABBIS2021}; in a related vein,
\citep{BBS2021-IEB} establishes finite convergence of alternating
projections under infeasibility and an error bound. Geometric
extensions to cones appear in~\citep{BBLOS2023,BC-Cones2026}. A
broader framework, in which CRM is a special case of a
circumcentered method induced by a family of isometries, is
developed by Bauschke, Ouyang, and
Wang~\citep{BauschkeOuyangWang2020}, with the corresponding
linear-convergence theory in~\citep{BauschkeOuyangWang2021c}.

Theorem~\ref{thm:main} also settles two questions raised in the
paper that introduced CRM. Behling, Bello-Cruz, and
Santos~\citeyearpar{BBS2018} prove the rate $c_F$ for $C_T$ on
$\RR^n$ and leave its sharpness open; and, after observing that
CRM beat the MAP rate $c_F^2$ in every numerical instance
reported there, they ask whether $c_F^2$ can be attained by CRM
in theory. Started in $V$, it can, with room to spare:
$\rho_V\le c_F^2/(2-c_F^2)<c_F^2$ whenever $\theta_F<\pi/2$
(Propositions~\ref{prop:strict} and~\ref{prop:worst-case}). The
constant $\rho_V$ is equally absent from the
circumcentered-isometry framework of Bauschke, Ouyang, and
Wang~\citeyearpar{BauschkeOuyangWang2020,BauschkeOuyangWang2021c}:
the two-subspace instances worked out there for the present
reflection order produce the rate $c_F$ and nothing depending on
$\theta_p$.

One-step convergence for $C_T$ has a clean geometric
characterization: $\rho_V=0$ if and only if $\theta_F=\theta_p$
(Corollary~\ref{cor:one-step}), covering the line,
codimension-one hyperplane \citep[Theorem~6.1]{BBS2020}, and
$V\subset U^\perp$ configurations; the corresponding termination
of $S_{\mu^\star}$ appears in~\citep[Remark~3.10]{BBNPW2016},
and $C_T$ achieves it without requiring the user to compute
$\mu^\star$.

A parallelized variant ($\pCRM$) of CRM, in which the composition
$R_VR_U$ is replaced by the independent reflections
$\{R_U,R_V\}$, was recently proposed by Barros, Behling, Guigues,
and Santos~\citeyearpar{BBGS2025-pCRM}. The operator $\pCRM$ lies
outside the scope of this paper, which focuses on the sequential
operator $C_T$ and the constant $\rho_V$.

\subsection{Complementary results and their purpose}\label{ssec:intro-complementary}

The results that follow Theorem~\ref{thm:main} are of two kinds.
The first kind completes the main line. The rate $\rho_V$ is the
exact contraction factor of $C_T$ on $V$, and its maximizers are
identified (Corollary~\ref{prop:rho-sharp}); no step along
$T(v)-v$, whatever rule selects it and whatever it knows about
the principal angles, does better
(Corollary~\ref{prop:single-step-opt}); the gap to the MAP rate
is quantified by the envelope $\rho_V\le c_F^2/(2-c_F^2)$
(Proposition~\ref{prop:worst-case}); one-step convergence occurs
exactly when $\theta_F=\theta_p$ (Corollary~\ref{cor:one-step});
and orbits confined to a narrower part of the spectrum of $M$
contract faster (Proposition~\ref{prop:local-rate}).

The second kind tests the structure of $C_T$, one feature at a
time. If $v_{k+1}$ may depend on $v_{k-1}$ as well as on $v_k$,
Chebyshev semi-iteration on $T$ attains
$\rho_{\mathrm{Cheb}}=(\sin\theta_p-\sin\theta_F)/(\sin\theta_p+\sin\theta_F)$,
smaller than $\rho_V$ by the factor
$(\sqrt a+\sqrt b)^2/(a+b)\in(1,2]$, with
$a\coloneqq\sin^2\theta_F$ and $b\coloneqq\sin^2\theta_p$, at the
price of two-step storage and knowledge of both angles
(Theorem~\ref{thm:cheb-strict}); this is what the single-step
restriction costs. Relaxing the reflections inside the
circumcenter does not help, because the relaxed vertices span
the same plane $W_v$ and $C_T(v)$ is already the point of that
plane closest to $\bar v$
(Proposition~\ref{prop:over-reflection-double}). Replacing the
adaptive circumcenter by the best fixed affine combination of
the three vertices gives, on $\RR^n$, a rate bounded below by
$1/3$; composed with $P_V$, the floor disappears and the rate is
exactly $\rho_V$, but only once the two coefficients are tuned to
both angles (Section~\ref{ssec:linear-competitor}, with proofs in
Online Resource~1). This is what the nonlinearity of $C_T$ buys.
Finally, Theorem~\ref{thm:main} assumes $v_0\in V$, and
Section~\ref{sec:outside-V} compares the ways of starting from a
general $x_0\in\RR^n$. Table~\ref{tab:contributions} collects
these results next to the known rates and serves as a map of the
paper.

\begin{table}[htbp]
\centering
\scriptsize
\renewcommand{\arraystretch}{1.6}
\setlength{\tabcolsep}{3pt}
\begin{threeparttable}
\begin{tabular}{l >{$}l<{$} l l}
\toprule
Method & \textrm{Iteration} & Worst-case rate (tuning) & Reference\\
\midrule
MAP & P_VP_U & $c_F^2$ \textrm{(pf)} & \tnote{a}\\
DRM & \tfrac12(\Id+R_VR_U) & $c_F$ \textrm{(pf)} & \tnote{b}\\
CRM on $\RR^n$ & C_T & $c_F$ \textrm{(pf)} & \tnote{c}\\
Relaxed AP & (1{-}\mu)\Id+\mu P_VP_U & $\rho_V$ at $\mu^\star{=}\tfrac{2}{s_F^2+s_p^2}$ \textrm{(t)} & \tnote{d}\\
Linesearch $B_T$ & \textrm{see }\eqref{eq:BT-our-order} & $\rho_V$ after one projection \textrm{(pf)} & \tnote{d}\\
GAP & (1{-}\alpha)\Id+\alpha\,P_V^{[\alpha_1]}P_U^{[\alpha_2]} & $\tfrac{1-s_F}{1+s_F}$ at $\alpha^\star_i{=}\tfrac{2}{1+s_F}$ \textrm{(t)} & \tnote{e}\\
AAMR & (1{-}\alpha)\Id+\alpha(2\beta P_V{-}\Id)(2\beta P_U{-}\Id) & $\tfrac{1-s_F}{1+s_F}$ at $\beta^\star{=}\tfrac{1}{1+s_F}$ \textrm{(t)} & \tnote{f}\\
\midrule
\multicolumn{4}{l}{\textbf{This paper}}\\[-1pt]
\midrule
CRM on $V$ & C_T=A_T=B_T \textrm{ on } V & $\rho_V=\tfrac{s_p^2-s_F^2}{s_p^2+s_F^2}$ \textrm{(\textbf{pf})} & Thms.~\ref{thm:main}, \ref{thm:linesearch}, \ref{thm:three-agree-on-V}\\
$\quad$sharpness & \textrm{worst-case ray } v^\star & $\sup$ attained, maximizers classified & Thm.~\ref{thm:rate}, Cor.~\ref{prop:rho-sharp}\\
$\quad$hierarchy & & $\rho_V\le\tfrac{c_F^2}{2-c_F^2}\le c_F^2$, sharp at $\theta_p=\tfrac{\pi}{2}$ & Props.~\ref{prop:strict}, \ref{prop:worst-case}\\
$\quad$one-step convergence & & $\rho_V=0\iff\theta_F=\theta_p$ & Cor.~\ref{cor:one-step}\\
Single-step class on $V$ & v+c(v)(T(v)-v) & $\ge\rho_V$: $C_T$ optimal in the class & Cor.~\ref{prop:single-step-opt}\\
Relaxed reflections & R_U^{(\lambda)},\,R_V^{(\eta)}\textrm{ inside }C_T & $\ge\rho_V$ pointwise; $=$ iff $(\lambda,\eta)=(1,1)$ & Prop.~\ref{prop:over-reflection-double}\\
Linear CDR family on $\RR^n$ & (1{-}\gamma{-}\beta)\Id{+}\beta R_VR_U{+}\gamma R_U & $\Gamma^\star\ge\tfrac13$ \textrm{(t)} & Sect.~\ref{ssec:linear-competitor}; OR1\\
Projected CDR family & C_{T,\gamma,\beta}\,P_V & $\rho_V$ at $\gamma{+}\beta{=}\tfrac{1}{s_F^2+s_p^2}$ \textrm{(t)} & Sect.~\ref{ssec:linear-competitor}; OR1\\
Chebyshev semi-iteration & \textrm{two-step memory} & $\rho_{\mathrm{Cheb}}=\tfrac{s_p-s_F}{s_p+s_F}<\rho_V$ \textrm{(t)} & Thm.~\ref{thm:cheb-strict}\\
Warm starts from $\RR^n$ & P_V, \textrm{ then } C_T & $\rho_V$ \textrm{(pf)} & Prop.~\ref{prop:strategy-A}; OR1\\
\bottomrule
\end{tabular}
\begin{tablenotes}
\footnotesize
\item[a] \citet{Aronszajn1950}; \citet{Kayalar1988};
\item[b] \citet{BCNPW2014};
\item[c] \citet{BBS2018,BBS2018-ORL};
\item[d] \citet{BBNPW2016};
\item[e] \citet{FaltGiselsson2017};
\item[f] \citet{AragonCampoy2018,AragonCampoy2019}.
\end{tablenotes}
\end{threeparttable}
\caption{Worst-case linear-convergence rates on two subspaces:
known results (upper block) and the results of this paper (lower
block). Notation: $s_F=\sin\theta_F$, $s_p=\sin\theta_p$,
$c_F=\cos\theta_F$, OR1 = Online Resource~1, and
$P_W^{[t]}\coloneqq(1-t)\Id+t\,P_W$ for the relaxed projection
onto a subspace $W$. Parameter-free (pf) means the worst-case
rate is achieved without prior knowledge of the principal
angles; tuned (t) means the optimal parameter requires the
Friedrichs angle
(and, for relaxed AP, projected CDR, and Chebyshev, also the
largest principal angle). For MAP, $c_F^2$ is the asymptotic power-norm ratio of
Fact~\ref{fact:friedrichs-norms}, not a per-step bound; $\rho_V$
holds at every step. The listed rates obey the hierarchy
$\rho_{\mathrm{Cheb}}<\rho_V\le c_F^2/(2-c_F^2)\le c_F^2\le
c_F<1$ of~\eqref{eq:rate-hierarchy} whenever
$\theta_F<\theta_p<\pi/2$. Two section-level summaries
complement this table: Table~\ref{tab:improvement} records which
structural relaxations of $C_T$ can and cannot improve on
$\rho_V$, and Table~\ref{tab:strategies} compares the warm-start
strategies of Section~\ref{sec:outside-V}.}
\label{tab:contributions}\label{tab:method-comparison}
\end{table}

\subsection{Organization}\label{ssec:intro-organization}

Section~\ref{sec:prelim} fixes notation and collects the
principal-angle material and the projection identity for $C_T$.
Section~\ref{sec:linesearch} carries out the three links of
Section~\ref{ssec:intro-main}: the relaxation formula, its
optimality along the MAP direction, and the sharp rate with its
consequences. Section~\ref{sec:strict-and-special} treats the
strict inequality $\rho_V<c_F^2$, one-step convergence, the
worst-case envelope at fixed $\theta_F$, and faster contraction
off the worst-case ray. Section~\ref{sec:improvement} examines
memory, relaxed reflections, and constant coefficients.
Section~\ref{sec:outside-V} analyzes initialization outside $V$.
Section~\ref{sec:numerics} gives the rate curves, an explicit
sharp instance in $\RR^4$, the verification of $\rho_V$ to unit
roundoff, and the comparison with AAMR.
Section~\ref{sec:conclusion} lists open questions. Online
Resource~1 contains the analysis of the linear two-parameter
family and of two further warm-start strategies.

\section{Preliminaries}\label{sec:prelim}

The following symbols recur throughout. $\RR^n$ is the ambient
Euclidean space, with inner product $\scal{\cdot}{\cdot}$ and
norm $\norm{\cdot}$. For a subspace $W\subseteq\RR^n$: $P_W$ is
the orthogonal projection onto $W$; $R_W=2P_W-\Id$ the
corresponding reflection; $W^\perp$ the orthogonal complement;
and $\dist(x,W)=\norm{x-P_Wx}$. $U$ and $V$ are the two
subspaces, with the standing assumption $V\not\subset U$; we
write $p=\dim V$, $s=\dim(U\cap V)\in\{0,\ldots,p-1\}$, and
$\bar x\coloneqq P_{U\cap V}(x)$, which specializes to
$\bar x_0=P_{U\cap V}(x_0)$ at the starting point and to
$\bar v=P_{U\cap V}(v)$ for $v\in V$, and satisfies
$\bar v_k=\bar x_0$ along the iteration of
Theorem~\ref{thm:main} by the projection-invariance property
established there. The principal angles between $U$ and $V$ are
$0\le\theta_1\le\cdots\le\theta_p\le\pi/2$
(Definition~\ref{def:principal-angles}), the first $s$ equal to
zero; the Friedrichs angle is
$\theta_F\coloneqq\theta_{s+1}\in(0,\pi/2]$
(Definition~\ref{def:friedrichs}), $c_F\coloneqq\cos\theta_F$,
and throughout we abbreviate
\[
   a\coloneqq\sin^2\theta_F,\qquad b\coloneqq\sin^2\theta_p.
\]
The principal-angle frame (Theorem~\ref{lem:frame}) is an
orthonormal basis $\{f_1,\ldots,f_p\}$ of $V$, paired with
orthonormal vectors $\{e_k\}\subset U$ for indices $k$ with
$\cos\theta_k>0$.

The basic operators are the MAP operator $T\coloneqq P_VP_U$,
the Douglas--Rachford operator
$T_{\mathrm{DR}}\coloneqq\tfrac12(\Id+R_VR_U)$, the relaxed-MAP
family $S_\mu\coloneqq(1-\mu)\Id+\mu T$ for $\mu\in\RR$, and the
circumcentered--reflection operator
\begin{equation}\label{eq:CT-def}
   C_T(x)\coloneqq\circum\{x,R_Ux,R_VR_Ux\},
\end{equation}
the main object of the paper. For each $x\in\RR^n$, the CRM
vertices, side vectors, and affine plane are
\begin{equation}\label{eq:y-z-W}
   y\coloneqq R_Ux,\quad z\coloneqq R_VR_Ux,
\end{equation}
\begin{equation}\label{eq:s-vectors}
   s_U\coloneqq y-x=2(P_Ux-x),\quad
   s_V\coloneqq z-x,\quad
   W_x\coloneqq\aff\{x,y,z\}.
\end{equation}
The adaptive step
$\mu_v\coloneqq\dist^2(v,U)/\norm{T(v)-v}^2$ is the one of
Theorem~\ref{thm:linesearch}, and the optimal constant relaxation
of $S_\mu$ is $\mu^\star=2/(a+b)$, with rate
$\rho_V=(b-a)/(a+b)$~\citep[Theorem~3.7]{BBNPW2016}. The letter
$\mu$ is reserved for relaxations of $S_\mu$, and $\mu_v$ is its
adaptive value. The self-adjoint operator
\begin{equation}\label{eq:M-definition}
   M\coloneqq(\Id-P_VP_U)|_V
\end{equation}
on $V$, with spectrum $\{\sin^2\theta_k\}_{k=1}^p$, drives the
rate analysis from Section~\ref{sec:rate} onward. Operators used
only in the comparisons of Sections~\ref{sec:improvement}
and~\ref{sec:outside-V} are defined where they first appear.

With notation in place, the rest of this section develops the
structural material we will need: the basic identities for
$C_T$; principal angles and the
Friedrichs angle; and the
principal-angle frame, the SVD of $P_U|_V$, and the operator
$M$. The frame
diagonalizes $P_VP_U$ on $V$ and reduces every later
computation to a trigonometric identity on a two-dimensional
invariant subspace.

The standing assumption $V\not\subset U$ makes the Friedrichs
angle positive and the rate $\rho_V$
in~\eqref{eq:mu-star-intro} smaller than~$1$; if $V\subset U$
then $\bar x_0=P_V(x_0)$ solves~\eqref{eq:best-approximation} in
a single projection. The assumption is deliberately one-sided:
$U\subset V$ is allowed and is covered by the analysis, with
$P_VP_U=P_U$, every positive principal angle equal to $\pi/2$,
and $\theta_F=\theta_p=\pi/2$, so that CRM started in $V$
converges in one step (Corollary~\ref{cor:one-step}). The
dimension of $U$ is arbitrary; only $p$ and $s$ enter the
analysis.

The finite-precision behavior of the $2\times2$ Gram system
behind the circumcenter of Definition~\ref{def:circum} is
analyzed in Remark~\ref{rem:conditioning}.

\begin{remark}[Convention on which space is iterated]\label{rem:convention}
We take $V$ as the iterate space, so that the iteration of
Theorem~\ref{thm:main} produces a sequence in $V$. The order of
composition matters: $C_T(x)=\circum\{x,R_Ux,R_VR_Ux\}$ applies
$R_U$ before $R_V$ and yields iterates in $V$ via the
linesearch identity of Theorem~\ref{thm:linesearch}. Swapping
the roles of $U$ and $V$ amounts to replacing the circumcenter
by $\circum\{x,R_Vx,R_UR_Vx\}$, which iterates the product
$P_UP_V$ and produces iterates in $U$ at the same sharp rate. The
quantities $\theta_F$, $\theta_p$, and $\rho_V$ are symmetric in
$(U,V)$, so the choice of labeling does not affect the rate; we
adopt the present convention because the asymmetric form
$T=P_VP_U$ produces iterates in $V$ directly via $C_T$. The
companion paper~\citep{BBNPW2016} adopts the reverse labeling,
$\dim U\le\dim V$, iteration in $U$, paired with the
order-swapped $B_T$ on the right of~\eqref{eq:BT-our-order}.
\end{remark}

The generalized circumcenter is fixed by the following
convention, which makes $C_T$ well-defined even in degenerate
vertex configurations.

\begin{definition}[Generalized circumcenter]\label{def:circum}
Let $S\subset\RR^n$ be a finite set such that some $s\in\RR^n$
is equidistant from all elements of $S$. We define
\begin{equation}\label{eq:circum-convention}
   \circum(S)  \coloneqq  P_{\aff(S)}(s),
\end{equation}
the orthogonal projection of $s$ onto the affine hull $\aff(S)$.
\end{definition}

The convention~\eqref{eq:circum-convention} produces a single
well-defined point, independent of the choice of equidistant $s$.

\begin{lemma}[Well-definedness of $\circum$]\label{lem:circum-well-defined}
If $s_1,s_2\in\RR^n$ are both equidistant from every element of
a finite set $S\subset\RR^n$, then
$P_{\aff(S)}(s_1)=P_{\aff(S)}(s_2)$.
\end{lemma}

\begin{proof}
Enumerate $S=\{p_0,p_1,\ldots,p_m\}$. The equidistance
hypothesis gives $\norm{s_i-p_0}=\norm{s_i-p_j}$ for $i=1,2$ and
$j=1,\ldots,m$, so subtracting the two squared equalities yields,
for $j=1,\ldots,m$,
\[
   \scal{s_1-s_2}{\,2p_0-2p_j}=\norm{p_0}^2-\norm{p_j}^2-(\norm{p_0}^2-\norm{p_j}^2)=0,
\]
i.e., $s_1-s_2\perp(p_j-p_0)$ for every $j$. The vectors
$\{p_j-p_0\}_{j=1}^m$ span the linear part of $\aff(S)$, so
$s_1-s_2\perp\aff(S)$. Orthogonal projections kill components
orthogonal to the target affine subspace, so
$P_{\aff(S)}(s_1)=P_{\aff(S)}(s_2)$.
\end{proof}

For the CRM configuration, the equidistance condition always
holds: any $s\in U\cap V$ is equidistant from $\{x,R_Ux,R_VR_Ux\}$,
since the reflections fix $U$ and $V$ pointwise. On
non-degenerate vertex triples, Definition~\ref{def:circum}
agrees with the classical Euclidean circumcenter; on degenerate
configurations (two vertices coinciding, or three distinct
collinear vertices), it returns the midpoint of the two extreme
distinct vertices. Both regimes satisfy the
Behling--Bello-Cruz--Santos identity~\citep[Lemma~3]{BBS2018}.

\begin{lemma}[BBS projection identity for $C_T$]\label{lem:BBS}
For every $x\in\RR^n$ and every $s\in U\cap V$,
\begin{equation}\label{eq:BBS}
   C_T(x)  =  P_{W_x}(s),
\end{equation}
or equivalently $C_T(x)\in W_x$ with
$s-C_T(x)\perp\Span\{s_U,s_V\}$. In particular,
\eqref{eq:BBS} characterizes $C_T(x)$ as the orthogonal
projection onto $W_x$ of any point in $U\cap V$.
\end{lemma}

\begin{proof}
The argument has two pieces: every $s\in U\cap V$ is equidistant
from the three vertices $\{x,y,z\}$, and the
well-definedness lemma~\ref{lem:circum-well-defined} then forces
$C_T(x)=P_{W_x}(s)$.

Let $s\in U\cap V$, so $P_Us=s$ and $P_Vs=s$, hence
$R_Us=R_Vs=s$ and $R_VR_Us=s$. Since $R_U$ is an isometry,
\[
   \norm{s-y}  =  \norm{R_Us-R_Ux}  =  \norm{s-x}.
\]
The same isometry argument with $R_VR_U$ in place of $R_U$ gives
$\norm{s-z}=\norm{s-x}$, so $s$ is equidistant from $\{x,y,z\}$.
By Definition~\ref{def:circum} applied to $S=\{x,y,z\}$ with any
equidistant point,
$C_T(x)=\circum\{x,y,z\}=P_{\aff\{x,y,z\}}(s)=P_{W_x}(s)$. The
right-hand side is independent of the choice of $s\in U\cap V$
by Lemma~\ref{lem:circum-well-defined}.

The point $C_T(x)=P_{W_x}(s)$ is the closest point to $s$ in the
affine subspace $W_x$, so the displacement $s-C_T(x)$ is
orthogonal to the linear part of $W_x$, which is
$W_x-x=\Span\{y-x,z-x\}=\Span\{s_U,s_V\}$. Hence
$s-C_T(x)\perp\Span\{s_U,s_V\}$.
\end{proof}

Identity~\eqref{eq:BBS} is used repeatedly in the rate
arguments below.

We use the Bj\"orck--Golub formulation of principal angles in
terms of the singular values of $P_U|_V$~\citep{BG1973},
equivalent to~\citep[Definition~3.1]{BBNPW2016}.

\begin{definition}[Principal angles]\label{def:principal-angles}
Let $\sigma_1\ge\cdots\ge\sigma_p\ge 0$ be the singular values of
$P_U|_V$ viewed as a linear map $V\to\RR^n$. The principal
angles between $V$ and $U$ are
\begin{equation}\label{eq:principal-angles}
   \theta_k\coloneqq\arccos\sigma_k\in[0,\pi/2]
   \quad(k=1,\ldots,p),
\end{equation}
ordered $0\le\theta_1\le\cdots\le\theta_p\le\pi/2$. Equivalently,
the cosines $\sigma_k=\cos\theta_k$ admit the recursive
variational characterization
\begin{equation}\label{eq:variational}
   \cos\theta_k
   =\max\bigl\{\scal{u}{v}:u\in U,\,v\in V,\,\norm{u}=\norm{v}=1,\
   u\perp u_1,\ldots,u_{k-1},\ v\perp v_1,\ldots,v_{k-1}\bigr\},
\end{equation}
with $(u_j,v_j)$ the maximizer at step $j$, and the convention
$\theta_k=\pi/2$ when $k>\dim U$ (no admissible $u$).
\end{definition}

The non-zero singular values of $P_U|_V$ and $P_V|_U$ coincide,
so the principal angles are symmetric in $(U,V)$ on their
non-trivial part.

\begin{definition}[Friedrichs angle]\label{def:friedrichs}
The Friedrichs angle between $U$ and $V$, introduced
in~\citep{Friedrichs1937}, is
\begin{equation}\label{eq:friedrichs-def}
   \theta_F\coloneqq\theta_{s+1}\in(0,\pi/2],\quad
   c_F\coloneqq\cos\theta_F\in[0,1).
\end{equation}
Equivalently~\citep[Definition~3.2]{BBNPW2016},
\begin{equation}\label{eq:friedrichs-variational}
   c_F=\sup\bigl\{|\scal{u}{v}|:u\in U\cap(U\cap V)^\perp,\
   v\in V\cap(U\cap V)^\perp,\
   \norm{u}\le 1,\ \norm{v}\le 1\bigr\}.
\end{equation}
The first $s$ principal angles are zero
(Theorem~\ref{lem:frame}), and
\citep[Proposition~3.3]{BBNPW2016} shows that $\theta_{s+1}$
matches the right-hand side of~\eqref{eq:friedrichs-variational}.
\end{definition}

The Friedrichs angle controls operator-norm convergence of the
classical projection methods.

\begin{fact}[{Friedrichs angle and operator norms;
\citealp[Theorem~9.31, Lemma~9.5]{Deutsch};
\citealp[Remark~3.5]{BBNPW2016};
\citealp[Theorem~4.1]{BCNPW2014}}]\label{fact:friedrichs-norms}
For every $k\in\NN$,
\begin{equation}\label{eq:friedrichs-norms}
   c_F=\norm{P_VP_U-P_{U\cap V}}=\norm{P_UP_V-P_{U\cap V}},
   \quad
   \norm{(P_VP_U)^k-P_{U\cap V}}=c_F^{2k-1}.
\end{equation}
The corresponding rate for the Douglas--Rachford operator is
$\norm{T_{\mathrm{DR}}-P_{\Fix T_{\mathrm{DR}}}}=c_F$.
\end{fact}

The largest principal angle $\theta_p$ satisfies
$\theta_F\le\theta_p\le\pi/2$, with both extremes possible. When
$\theta_F=\theta_p$, $C_T$ converges in a single step
(Section~\ref{sec:special}); when $\theta_p=\pi/2$, which
requires $V\cap U^\perp\neq\{0\}$, the rate attains its worst
case for fixed $\theta_F$ (Proposition~\ref{prop:worst-case}).

The operator $M=(\Id-P_VP_U)|_V$ introduced
in~\eqref{eq:M-definition} is the workhorse of the rate
analysis. It is self-adjoint on $V$, with spectrum
$\{\sin^2\theta_k\}_{k=1}^p$; its eigenstructure, shared with
$P_VP_U|_V$ and $P_U|_V$, is the principal-angle frame, which we
now identify as the SVD of $P_U|_V$. The classical
numerical-linear-algebra setting for this material is the
cosine--sine (CS) decomposition of orthogonal
projections~\citep{PaigeWei1994,StewartSun}; we use the equivalent
vector form, which streamlines the rate analysis.

\begin{theorem}[Principal-angle frame as SVD of $P_U|_V$]\label{lem:frame}
The map $P_U|_V\colon V\to U$ admits a singular value
decomposition. There exists an orthonormal basis
$\{f_1,\ldots,f_p\}$ of $V$ and an orthonormal family
$\{e_k\}_{k\in K}\subset U$, indexed by
$K\coloneqq\{k\in\{1,\ldots,p\}:\theta_k<\pi/2\}$ with
$|K|\le\min(\dim U,p)$, such that
\begin{equation}\label{eq:frame}
   P_Uf_k=\cos\theta_k\,e_k\quad(k\in K),
   \quad
   P_Uf_k=0\quad(k\notin K),
   \quad
   P_Ve_k=\cos\theta_k\,f_k\quad(k\in K),
\end{equation}
and $\scal{e_j}{f_k}=\cos\theta_k\,\delta_{jk}$ for $j,k\in K$.
The cosines $\cos\theta_1\ge\cdots\ge\cos\theta_p\ge 0$ are the
singular values of $P_U|_V$. The extreme angles correspond to
extreme geometric configurations:
\begin{equation}\label{eq:frame-extremes}
   \theta_k=0  \Longleftrightarrow  f_k=e_k\in U\cap V,
   \quad
   \theta_k=\pi/2  \Longleftrightarrow  k\notin K\text{ and }f_k\in V\cap U^\perp.
\end{equation}
This vector form is equivalent to the matrix block
decomposition~\citep[Proposition~3.4]{BBNPW2016} and to the CS
decomposition of the pair
$(P_U,P_V)$~\citep{PaigeWei1994,StewartSun}.
\end{theorem}

\begin{proof}
The linear map $P_U|_V\colon V\to U$ has a singular value
decomposition: there exist an orthonormal basis
$\{f_1,\ldots,f_p\}$ of $V$ and orthonormal vectors
$\{e_k\}_{k\in K}$ in $U$, with $K=\{k:\sigma_k>0\}$, such that
\begin{equation}\label{eq:svd-step}
   P_Uf_k=\sigma_ke_k\quad(k\in K),
   \quad
   P_Uf_k=0\quad(k\notin K),
\end{equation}
where $\sigma_1\ge\cdots\ge\sigma_p\ge 0$ are the singular
values. The substitution $\sigma_k=\cos\theta_k$, valid by
Definition~\ref{def:principal-angles}, identifies $K$ with
$\{k:\theta_k<\pi/2\}$ and gives the first two identities
in~\eqref{eq:frame}. Since $\{e_k\}_{k\in K}$ is orthonormal in
$U$, $|K|\le\dim U$.

For the third identity in~\eqref{eq:frame}, fix $k\in K$. Since
$f_j\in V$ for every $j$,
$P_Ve_k=\sum_{j=1}^p\scal{P_Ve_k}{f_j}f_j=\sum_j\scal{e_k}{f_j}f_j$
by self-adjointness of $P_V$. We compute the inner products.
For $j\in K$, $e_j\in U$ implies $P_Ue_j=e_j$ and
self-adjointness of $P_U$ gives
$\scal{e_k}{f_j}=\scal{P_Ue_k}{f_j}=\scal{e_k}{P_Uf_j}=\sigma_j\scal{e_k}{e_j}=\sigma_j\delta_{jk}=\sigma_k\delta_{jk}$.
For $j\notin K$, $P_Uf_j=0$ gives
$\scal{e_k}{f_j}=\scal{P_Ue_k}{f_j}=\scal{e_k}{P_Uf_j}=0$.
Hence $P_Ve_k=\sigma_kf_k=\cos\theta_k\,f_k$.

For~\eqref{eq:frame-extremes}: $\theta_k=0$ means $\sigma_k=1$,
so $\norm{P_Uf_k}=\norm{f_k}=1$ and Pythagoras forces
$f_k=P_Uf_k\in U$; combined with $f_k\in V$ this gives
$f_k\in U\cap V$, and~\eqref{eq:svd-step} reads $f_k=e_k$. The
converse is immediate. The case $\theta_k=\pi/2$ is equivalent
to $\sigma_k=0$, i.e.\ $k\notin K$, equivalently $P_Uf_k=0$,
i.e.\ $f_k\in V\cap U^\perp$.
\end{proof}

\begin{lemma}[The operator $M$]\label{lem:M}
The operator $M=(\Id-P_VP_U)|_V$ defined
in~\eqref{eq:M-definition} maps $V$ into $V$ and is self-adjoint
on $V$. For every $v\in V$,
\begin{equation}\label{eq:M-quadratic}
   Mv=v-T(v),\quad
   \scal{v}{Mv}=\dist^2(v,U),\quad
   \scal{v}{M^2v}=\norm{v-T(v)}^2.
\end{equation}
\end{lemma}

\begin{proof}
For $v\in V$, $P_Vv=v$ gives $Mv=v-P_VP_Uv=v-T(v)\in V$. For
self-adjointness on $V$, take $v,w\in V$ and use $P_Vv=v$,
$P_Vw=w$, together with self-adjointness of $P_U$ and $P_V$:
\[
   \scal{Mv}{w}=\scal{v-P_VP_Uv}{w}=\scal{v}{w}-\scal{P_Uv}{w},
\]
and the same computation with $v$ and $w$ swapped gives
$\scal{v}{Mw}=\scal{v}{w}-\scal{P_Uv}{w}$. The first quadratic
form is
$\scal{v}{Mv}=\norm{v}^2-\scal{v}{P_Uv}=\norm{v}^2-\norm{P_Uv}^2=\dist^2(v,U)$,
where the second equality uses idempotence of $P_U$ and the
third the Pythagorean identity. The second quadratic form is
$\scal{v}{M^2v}=\scal{Mv}{Mv}=\norm{v-T(v)}^2$.
\end{proof}

\begin{lemma}[Eigendecomposition of $M$]\label{lem:spectrum}
With the principal-angle frame of Theorem~\ref{lem:frame},
\begin{equation}\label{eq:M-eigen}
   Mf_k  =  \sin^2\theta_k\,f_k
   \quad(k=1,\ldots,p).
\end{equation}
The spectrum of $M$ on $V$ is $\{\sin^2\theta_k:k=1,\ldots,p\}$,
its kernel is
\[
   \ker M|_V  =  U\cap V  =  \Span\{f_1,\ldots,f_s\},
\]
and the spectrum of $M|_{V\cap(U\cap V)^\perp}$
is~$\{\sin^2\theta_{s+1},\ldots,\sin^2\theta_p\}$, with extremes
$\sin^2\theta_F$ and $\sin^2\theta_p$.
\end{lemma}

\begin{proof}
By Theorem~\ref{lem:frame}, $P_Uf_k=\cos\theta_k\,e_k$ and
$P_Ve_k=\cos\theta_k\,f_k$, hence
\[
   Tf_k  =  P_VP_Uf_k  =  \cos\theta_k\,P_Ve_k  =  \cos^2\theta_k\,f_k,
\]
so $Mf_k=f_k-Tf_k=\sin^2\theta_k\,f_k$. The spectrum of $M|_V$ is
therefore $\{\sin^2\theta_1,\ldots,\sin^2\theta_p\}$.

The kernel: $Mf_k=0$ if and only if $\theta_k=0$, equivalently
$k\le s$ by Definition~\ref{def:friedrichs}; for these indices,
$f_k\in U\cap V$ by Theorem~\ref{lem:frame}. Conversely, every
$u\in U\cap V$ satisfies $P_Uu=u$ and $P_Vu=u$, so $Mu=0$.
Hence $\ker M|_V=U\cap V=\Span\{f_1,\ldots,f_s\}$, and the
orthogonal complement of the kernel inside $V$ is
$\Span\{f_{s+1},\ldots,f_p\}$, on which $M$ acts diagonally with
eigenvalues $\{\sin^2\theta_{s+1},\ldots,\sin^2\theta_p\}$.
\end{proof}

\begin{remark}[$M$ as the central object of the rate analysis]\label{rem:M-central}
The Kantorovich-type bound that drives Theorem~\ref{thm:main} is
a two-moment inequality at $w=v-\bar v$ involving only
$\scal{w}{Mw}$ and $\scal{w}{M^2w}$; by Lemma~\ref{lem:M}, both
are computable from the geometry as $\dist^2(v,U)$ and
$\norm{v-T(v)}^2$ without ever diagonalizing $M$. The
linesearch step in Theorem~\ref{thm:linesearch} below is the
ratio of the two, recovering the optimal contraction parameter
$\mu^\star=2/(a+b)$ adaptively.
\end{remark}

\section{From the circumcenter to the sharp rate}\label{sec:linesearch}

This section carries out the three links of
Section~\ref{ssec:intro-main}. The first, the projection identity
of Lemma~\ref{lem:BBS}, is already in place. What remains is to
identify the plane $W_v$ for $v\in V$, to project onto it, and to
estimate the result. Throughout, $T=P_VP_U$, so that the
projection onto $U$ is applied first, and $M=(\Id-P_VP_U)|_V$ is
the operator of~\eqref{eq:M-definition}.

\subsection{The circumcenter on $V$ is a relaxed alternating-projection step}\label{ssec:relaxed-step}

\begin{lemma}[The plane of the CRM vertices]\label{lem:plane}
Let $T=P_VP_U$, let $v\in V$, and let $s_U,s_V$ be the side
vectors of~\eqref{eq:s-vectors}. Then
\begin{equation}\label{eq:sides-basis}
   s_U=-2\bigl(Mv-P_{V^\perp}P_Uv\bigr),\qquad
   s_V=-2\bigl(Mv+P_{V^\perp}P_Uv\bigr),
\end{equation}
with $Mv=v-T(v)\in V$ and $P_{V^\perp}P_Uv\in V^\perp$. In
particular, $W_v=v+\Span\{Mv,P_{V^\perp}P_Uv\}$, the two spanning
vectors are orthogonal, and
\begin{equation}\label{eq:pythag-V}
   \dist^2(v,U)=\scal{v}{Mv}=\norm{Mv}^2+\norm{P_{V^\perp}P_Uv}^2.
\end{equation}
\end{lemma}

\begin{proof}
Since $v\in V$, we have
$P_Uv=P_VP_Uv+P_{V^\perp}P_Uv=(v-Mv)+P_{V^\perp}P_Uv$, so
$s_U=2(P_Uv-v)=-2(Mv-P_{V^\perp}P_Uv)$. Applying
$R_V=\Id-2P_{V^\perp}$ to $y=v+s_U$ and using
$P_{V^\perp}v=P_{V^\perp}Mv=0$ gives $z=y-4P_{V^\perp}P_Uv$,
hence $s_V=s_U-4P_{V^\perp}P_Uv=-2(Mv+P_{V^\perp}P_Uv)$. The
change of basis in~\eqref{eq:sides-basis} is invertible, so
$\Span\{s_U,s_V\}=\Span\{Mv,P_{V^\perp}P_Uv\}$ and
$W_v=v+\Span\{s_U,s_V\}$ has the stated form; the two vectors are
orthogonal because $Mv\in V$ and $P_{V^\perp}P_Uv\in V^\perp$.
For~\eqref{eq:pythag-V}, self-adjointness of $P_V$ and $P_Vv=v$
give
$\scal{v}{Mv}=\norm{v}^2-\scal{v}{P_Uv}=\norm{v}^2-\norm{P_Uv}^2=\dist^2(v,U)$,
and the orthogonal decomposition $v-P_Uv=Mv-P_{V^\perp}P_Uv$
gives $\dist^2(v,U)=\norm{Mv}^2+\norm{P_{V^\perp}P_Uv}^2$.
\end{proof}

\begin{theorem}[Linesearch formula on $V$]\label{thm:linesearch}
Let $T=P_VP_U$, let $v\in V$, and let $\bar v=P_{U\cap V}(v)$.
Either
\begin{equation}\label{eq:linesearch-degenerate}
   T(v)=v
   \quad
   (\text{in which case }v\in U\cap V\text{ and }C_T(v)=v=\bar v),
\end{equation}
or $T(v)\neq v$ and
\begin{equation}\label{eq:linesearch}
   C_T(v)=v+\mu_v\bigl(T(v)-v\bigr),
   \quad
   \mu_v\coloneqq\frac{\dist^2(v,U)}{\norm{T(v)-v}^2}.
\end{equation}
In either case $C_T(v)\in V$.
\end{theorem}

\begin{proof}
If $T(v)=v$, then $Mv=0$ and~\eqref{eq:pythag-V} gives
$\dist(v,U)=0$; hence $v\in U\cap V$, $\bar v=v$, the three
vertices coincide with $v$, and $C_T(v)=v$. Suppose $T(v)\ne v$.
Lemma~\ref{lem:BBS} with $s=\bar v$, together with
Lemma~\ref{lem:plane}, gives $C_T(v)=v+P_E(\bar v-v)$, where
$E\coloneqq\Span\{Mv,P_{V^\perp}P_Uv\}$ is spanned by the
orthogonal vectors $Mv\ne0$ and $P_{V^\perp}P_Uv$ (the latter
possibly zero). The vector $\bar v-v$ lies in $V$, so it is
orthogonal to $P_{V^\perp}P_Uv$, and
$\scal{\bar v-v}{Mv}=-\scal{v}{Mv}$ because $M$ is self-adjoint
on $V$ with $M\bar v=0$ (Lemmas~\ref{lem:M}
and~\ref{lem:spectrum}). Therefore
\[
   C_T(v)=v-\frac{\scal{v}{Mv}}{\norm{Mv}^2}\,Mv
   =v+\frac{\dist^2(v,U)}{\norm{T(v)-v}^2}\bigl(T(v)-v\bigr)
\]
by~\eqref{eq:pythag-V} and $Mv=v-T(v)$. Both $v$ and $T(v)$ lie
in $V$, so $C_T(v)\in V$.
\end{proof}

\begin{remark}[Relation to known circumcenter formulas]\label{rem:known-formulas}
Formula~\eqref{eq:linesearch} is the two-subspace instance of a
known representation. For a closed convex set $A$ and an affine
subspace $B$ with $A\cap B\ne\emptyset$, the CRM point computed
from $x\in B$ is $P_{H(x)\cap B}(x)$ for a closed convex set
$H(x)\supseteq A$, a halfspace when $x\notin A$
\citep[Proposition~2.3(ii)]{Behling-SIOPT2024}. For $A=U$, $B=V$,
and $v\in V\setminus U$, take
$H(v)\coloneqq\{y:\scal{y-P_Uv}{v-P_Uv}\le0\}$, the halfspace
supporting $U$ at $P_Uv$. For $y\in V$, the identities
$\scal{P_Uv}{v-P_Uv}=0$ and $P_V(v-P_Uv)=Mv$ give
$H(v)\cap V=\{y\in V:\scal{y}{Mv}\le0\}$, and projecting $v$ onto
this half-subspace returns
$v-(\scal{v}{Mv}/\norm{Mv}^2)\,Mv$, which is~\eqref{eq:linesearch}.
In Pierra's product space, \citep[Lemma~4.2]{Behling-SIOPT2024}
shows that the CRM step for the diagonal subspace and a
separating set that changes with the iteration, and need not be
a subspace, is the extrapolated simultaneous projection
$x^{k+1}=x^k-\alpha_kw^k$
of~\citep[(3.5)--(3.6)]{Behling-SIOPT2024}; its coefficient
$\alpha_k$ is the mean squared projection residual divided by the
squared norm of the averaged residual, and $\mu_v$ is the same
quotient for the pair $(U,V)$, with the average replaced by $P_V$.
What the linear setting adds is that $T(v)-v=-M(v-\bar v)$ is a
fixed self-adjoint operator applied to the error. Then
$\mu_v=\scal{w}{Mw}/\norm{Mw}^2$ with $w=v-\bar v$, and the
one-step contraction becomes the Kantorovich-type
ratio~\eqref{eq:residual} of Section~\ref{sec:rate}. The next
corollary is the counterpart, for $\{S_\mu\}$ on $V$,
of~\citep[Lemma~4.4]{BBNPW2016} for the partial-relaxation family
on $U$.
\end{remark}

\begin{corollary}[$\mu_v$ is the error-minimizing relaxation]\label{lem:mu-optimal}
Let $v\in V$ with $T(v)\ne v$. Among the relaxed steps
$S_\mu(v)=v+\mu(T(v)-v)$, $\mu\in\RR$, of~\eqref{eq:S-mu-intro},
the one closest to $\bar v$ is $C_T(v)=S_{\mu_v}(v)$, and $\mu_v$
is the unique minimizer of $\mu\mapsto\norm{S_\mu(v)-\bar v}$. In
particular,
\begin{equation}\label{eq:CT-vs-S-mu}
   \norm{C_T(v)-\bar v}\le\norm{S_\mu(v)-\bar v}
   \quad(\mu\in\RR),
\end{equation}
with equality only at $\mu=\mu_v$.
\end{corollary}

\begin{proof}
The line $\{S_\mu(v):\mu\in\RR\}=v+\RR\,Mv$ lies in $W_v$ by
Lemma~\ref{lem:plane} and contains $C_T(v)$ by
Theorem~\ref{thm:linesearch}. By Lemma~\ref{lem:BBS}, $C_T(v)$
is the point of $W_v$ closest to $\bar v$, hence also the point
of the line closest to $\bar v$. The minimizer is unique because
$Mv\ne0$.
\end{proof}

The linesearch step $\mu_v$ is what the projection-based maps
$A_T$ and $B_T$ of~\eqref{eq:AT-our-order}--\eqref{eq:BT-our-order}
also produce on $V$.

\begin{theorem}[The three linesearch maps agree on $V$]\label{thm:three-agree-on-V}
For every $v\in V$,
\begin{equation}\label{eq:three-agree}
   A_T(v)  =  B_T(v)  =  C_T(v).
\end{equation}
The common value is that of Theorem~\ref{thm:linesearch}: when
$T(v)\ne v$,
\begin{equation}\label{eq:three-agree-formula}
   A_T(v)=B_T(v)=C_T(v)=v+\mu_v(T(v)-v),
   \quad
   \mu_v=\frac{\dist^2(v,U)}{\norm{T(v)-v}^2}.
\end{equation}
\end{theorem}

\begin{proof}
For $v\in V$, $P_Vv=v$ makes the outer projection
in~\eqref{eq:BT-our-order} trivial, so the formulas for
$A_T(v)$ and $B_T(v)$ collapse to
$(1-\lambda_v)v+\lambda_v P_VP_Uv$ with the common step
\[
   \lambda_v=\mu_v^{\mathrm B}=\frac{\scal{v-P_VP_Uv}{v}}{\norm{v-P_VP_Uv}^2}=\frac{\scal{v}{Mv}}{\norm{Mv}^2},
\]
using $v-T(v)=Mv$ (Lemma~\ref{lem:M}). The numerator equals
$\dist^2(v,U)$ and the denominator equals $\norm{T(v)-v}^2$,
again by Lemma~\ref{lem:M}; thus
$\lambda_v=\mu_v^{\mathrm B}=\mu_v$,
and~\eqref{eq:three-agree-formula} follows from
Theorem~\ref{thm:linesearch}.

When $T(v)=v$, $P_VP_Uv=v$, the convention
$\lambda_v=\mu_v^{\mathrm B}=1$ gives $A_T(v)=B_T(v)=v$, and
Theorem~\ref{thm:linesearch} gives $C_T(v)=v$.
\end{proof}

\noindent
The three maps coincide on $V$ and therefore share the rate
$\rho_V$ there. Off $V$ they differ; see
Section~\ref{ssec:off-V}.

\subsection{The sharp rate via Kantorovich's inequality}\label{sec:rate}

Combining the linesearch formula~\eqref{eq:linesearch} with the
operator $M$ of \eqref{eq:M-definition} gives a
compact form for $C_T(v)$ that prepares the Kantorovich-type
estimate of the next section. For $v\in V$ with $Mv\ne 0$
(equivalently $T(v)\ne v$),
\begin{equation}\label{eq:CT-via-M}
   C_T(v)  =  v  -  \frac{\scal{v-\bar v}{M(v-\bar v)}}{\scal{v-\bar v}{M^2(v-\bar v)}}\,M(v-\bar v),
\end{equation}
using $T(v)-v=-Mv$ from Lemma~\ref{lem:M} and
$M(v-\bar v)=Mv$ (since $M\bar v=0$ by
Lemma~\ref{lem:spectrum}). The ratio in~\eqref{eq:CT-via-M} is
the optimal step length: it recovers the relaxation parameter
$\mu^\star=2/(a+b)$ of the linear $S_\mu$-family at the
worst-case ray, with $a$ and $b$ the extreme nonzero eigenvalues
of $M|_V$.

The rate inequality~\eqref{eq:main-rate} follows from
Theorem~\ref{thm:three-agree-on-V} and the on-$V$ specialization
of~\citep[Theorem~4.5]{BBNPW2016}; this section gives an
independent spectral proof via Kantorovich's inequality applied
to $M=(\Id-P_VP_U)|_V$. The direct argument identifies the
worst-case ray $v^\star$ as the Kantorovich extremizer, gives
the exact contraction factor of
Corollary~\ref{prop:rho-sharp}, and makes the strict
inequality $\rho_V<c_F^2$ (Proposition~\ref{prop:strict}) and
the worst-case envelope $\rho_V\le c_F^2/(2-c_F^2)$
(Proposition~\ref{prop:worst-case}) visible as algebraic
consequences. It is self-contained, with no appeal to the
operator-theoretic machinery of~\citep[Section~2]{BBNPW2016}.

We start by recalling the form of Kantorovich's
inequality~\citep{Kantorovich1948} we use: for a self-adjoint
positive-semidefinite $A$ with nonzero spectrum in $[a,b]$, the
ratio $\scal{x}{Ax}^2/(\norm{x}^2\scal{x}{A^2x})$ is bounded
below by a function of $a,b$ alone. The proof is a one-line
consequence of $(\lambda-a)(\lambda-b)\le 0$ on $[a,b]$; we
include it for self-containment.

\begin{lemma}[Kantorovich]\label{lem:kantorovich}
Let $A$ be a self-adjoint positive-semidefinite operator on a
Euclidean space, with spectrum contained in $\{0\}\cup[a,b]$,
$0<a\le b$. For every nonzero $x\in(\ker A)^\perp$,
\begin{equation}\label{eq:kantorovich}
   \frac{\scal{x}{Ax}^2}{\norm{x}^2\,\scal{x}{A^2x}}
   \ge\frac{4ab}{(a+b)^2}.
\end{equation}
If both $a$ and $b$ are eigenvalues of $A$, with corresponding
unit eigenvectors $u_a$ and $u_b$, then equality holds at
$x^\star=\sqrt{b}\,u_a+\sqrt{a}\,u_b$.
\end{lemma}

\begin{proof}
On $(\ker A)^\perp$ the spectrum of $A$ lies in $[a,b]$, so the
spectral theorem gives an orthonormal eigenbasis
$\{u_i\}_{i\in I}$ of $(\ker A)^\perp$ with eigenvalues
$\lambda_i\in[a,b]$. Write $x=\sum_i\xi_iu_i$ with
$\xi_i=\scal{x}{u_i}$; since $x\ne 0$,
$\norm{x}^2=\sum_i\xi_i^2>0$. Set the probability weights
$\pi_i\coloneqq\xi_i^2/\norm{x}^2$ and the moments
\[
   m\coloneqq\sum_i\pi_i\lambda_i\in[a,b],\quad
   q\coloneqq\sum_i\pi_i\lambda_i^2\in[a^2,b^2].
\]
Then $\scal{x}{Ax}=\norm{x}^2m$ and
$\scal{x}{A^2x}=\norm{x}^2q$, so the left-hand side
of~\eqref{eq:kantorovich} is $m^2/q$, and $q\ge a^2>0$ makes it
well defined.

Every $\lambda_i\in[a,b]$ satisfies
$(\lambda_i-a)(\lambda_i-b)\le 0$, i.e.,
$\lambda_i^2\le(a+b)\lambda_i-ab$. Taking the $\pi$-expectation,
\begin{equation}\label{eq:second-moment-bound}
   q\le(a+b)m-ab.
\end{equation}
Substituting into $(a+b)^2m^2-4abq$ gives a perfect square:
\[
   (a+b)^2m^2-4abq\ge(a+b)^2m^2-4ab[(a+b)m-ab]=[(a+b)m-2ab]^2\ge 0.
\]
Rearranging and dividing by $q(a+b)^2>0$
yields~\eqref{eq:kantorovich}.

For the equality case, take $x^\star=\sqrt{b}\,u_a+\sqrt{a}\,u_b$
as in the statement. Then $\norm{x^\star}^2=a+b$, and
\[
   \scal{x^\star}{Ax^\star}=ab+ab=2ab,\quad
   \scal{x^\star}{A^2x^\star}=a^2b+ab^2=ab(a+b),
\]
so
$\dfrac{\scal{x^\star}{Ax^\star}^2}{\norm{x^\star}^2\scal{x^\star}{A^2x^\star}}
=\dfrac{4a^2b^2}{(a+b)\cdot ab(a+b)}=\dfrac{4ab}{(a+b)^2}$.
\end{proof}

With Kantorovich's inequality in hand, the proof of
Theorem~\ref{thm:main} reduces to recognizing that the ratio
$\norm{C_T(v)-\bar v}^2/\norm{v-\bar v}^2$ is, after a single
substitution, a Kantorovich-type ratio for the operator $A=M|_V$.

\begin{theorem}[Per-step contraction of $C_T$ on $V$]\label{thm:rate}
For every $v\in V$ and $\bar v=P_{U\cap V}(v)$,
\begin{equation}\label{eq:rate-bound}
   \norm{C_T(v)-\bar v}  \le  \rho_V\,\norm{v-\bar v},
   \quad
   \rho_V  =  \frac{\sin^2\theta_p-\sin^2\theta_F}{\sin^2\theta_p+\sin^2\theta_F}.
\end{equation}
The bound is attained: at the worst-case ray
\begin{equation}\label{eq:worst-ray}
   v^\star  \coloneqq  \bar v+\sin\theta_p\,f_{s+1}+\sin\theta_F\,f_p,
\end{equation}
$\norm{C_T(v^\star)-\bar v}=\rho_V\,\norm{v^\star-\bar v}$. In
particular, iterating $v_{k+1}=C_T(v_k)$ from any $v_0\in V$
yields Q-linear convergence to $P_{U\cap V}(v_0)$ with rate at
most $\rho_V$, sharp at any $v_0$ on the worst-case ray.
\end{theorem}

\begin{proof}
Set $w\coloneqq v-\bar v\in V$. The operator $C_T$ commutes with
translation by elements of $U\cap V$: both $R_U$ and $R_V$ fix
$U\cap V$ pointwise, so for $u\in U\cap V$ the three CRM
vertices of $v+u$ are those of $v$ shifted by $u$, and their
circumcenter shifts by $u$ as well. We may therefore assume
$\bar v=0$ without loss of generality, in which case $w=v$.

By Lemma~\ref{lem:M}, $Mw=v-T(v)$. If $Mw=0$ then $T(v)=v$ and
Theorem~\ref{thm:linesearch} gives $C_T(v)=v=\bar v$, so
\eqref{eq:rate-bound} is trivial; assume $Mw\ne 0$. The
condition $\bar v=P_{U\cap V}(v)=0$ places $w$ in
$V\cap(U\cap V)^\perp=\Span\{f_{s+1},\ldots,f_p\}$.

By~\eqref{eq:CT-via-M},
\[
   C_T(w)=w-\mu\,Mw,\quad
   \mu=\frac{\scal{w}{Mw}}{\scal{w}{M^2w}}.
\]
Expanding
$\norm{w-\mu Mw}^2=\norm{w}^2-2\mu\scal{w}{Mw}+\mu^2\scal{w}{M^2w}$
and using $\mu\scal{w}{M^2w}=\scal{w}{Mw}$,
\begin{equation}\label{eq:residual}
   \frac{\norm{C_T(w)}^2}{\norm{w}^2}
   =1-\frac{\scal{w}{Mw}^2}{\norm{w}^2\,\scal{w}{M^2w}}.
\end{equation}

By Lemmas~\ref{lem:M} and~\ref{lem:spectrum}, $M|_V$ is
self-adjoint positive-semidefinite with spectrum in
$\{0\}\cup[a,b]$ and $\ker(M|_V)=U\cap V$. Since
$w\in(\ker M|_V)^\perp\setminus\{0\}$,
Lemma~\ref{lem:kantorovich} applied to $A=M|_V$ gives
\begin{equation}\label{eq:kant-applied}
   \frac{\scal{w}{Mw}^2}{\norm{w}^2\,\scal{w}{M^2w}}\ge\frac{4ab}{(a+b)^2}=1-\rho_V^2,
\end{equation}
the algebraic identity using $4ab=(a+b)^2-(b-a)^2$ and
$\rho_V=(b-a)/(a+b)$.
Combining~\eqref{eq:residual} and~\eqref{eq:kant-applied}
yields $\norm{C_T(w)}^2\le\rho_V^2\norm{w}^2$, which
is~\eqref{eq:rate-bound}.

For sharpness, take the Kantorovich extremizer
$x^\star=\sqrt{b}\,u_a+\sqrt{a}\,u_b$ with $u_a=f_{s+1}$ and
$u_b=f_p$. Using $\sqrt{\sin^2\theta}=\sin\theta$ on
$[0,\pi/2]$, this becomes
$x^\star=\sin\theta_p\,f_{s+1}+\sin\theta_F\,f_p\in
V\cap(U\cap V)^\perp\setminus\{0\}$; translating back,
$v^\star=\bar v+x^\star$ satisfies~\eqref{eq:rate-bound} with
equality.
\end{proof}

The rate is exact, and the maximizers can be read off from the
equality case of Kantorovich's inequality.

\begin{corollary}[Exact contraction factor]\label{prop:rho-sharp}
The contraction factor of $C_T$ on $V$ relative to $U\cap V$ is
$\rho_V$:
\begin{equation}\label{eq:rho-sharp}
   \sup_{v\in V\setminus(U\cap V)}
      \frac{\norm{C_T(v)-P_{U\cap V}(v)}}{\norm{v-P_{U\cap V}(v)}}
     =  \rho_V.
\end{equation}
If $\theta_F<\theta_p$, the supremum is attained exactly at the
points $v$ whose error $w=v-P_{U\cap V}(v)$ splits as
$w=w_a+w_b$ with $Mw_a=a\,w_a$, $Mw_b=b\,w_b$, and
$a\norm{w_a}^2=b\norm{w_b}^2>0$; the worst-case
ray~\eqref{eq:worst-ray} is one such family. If
$\theta_F=\theta_p$, then $\rho_V=0$ and $C_T(v)=P_{U\cap V}(v)$
for every $v\in V$.
\end{corollary}

\begin{proof}
The supremum is at most $\rho_V$ by Theorem~\ref{thm:rate} and is
attained on~\eqref{eq:worst-ray}. By~\eqref{eq:residual}, a point
$v\notin U\cap V$ attains it if and only if $w$ gives equality
in~\eqref{eq:kant-applied}. In the notation of the proof of
Lemma~\ref{lem:kantorovich}, this requires equality
in~\eqref{eq:second-moment-bound}, which puts all spectral weight
of $w$ on the eigenvalues $a$ and $b$, and $(a+b)m=2ab$, which
for $a<b$ fixes these weights at $b/(a+b)$ and $a/(a+b)$; this is
the condition $a\norm{w_a}^2=b\norm{w_b}^2$. Conversely, for such
$w$ both inequalities in that proof are equalities. For $a=b$,
$\rho_V=0$ and Theorem~\ref{thm:rate} gives $C_T(v)=\bar v$.
\end{proof}

No other choice of step along the same direction does better,
whatever it is allowed to know.

\begin{corollary}[No step along $T(v)-v$ beats $\rho_V$]\label{prop:single-step-opt}
Let $c\colon V\to\RR$ be any function, and let
$\Phi_c(v)\coloneqq v+c(v)\bigl(T(v)-v\bigr)$. Then
\begin{equation}\label{eq:single-step-opt}
   \norm{\Phi_c(v)-\bar v}\ge\norm{C_T(v)-\bar v}
   \quad(v\in V),
   \qquad
   \sup_{v\in V\setminus(U\cap V)}
   \frac{\norm{\Phi_c(v)-\bar v}}{\norm{v-\bar v}}\ge\rho_V.
\end{equation}
The rule $c$ may use any information, including the principal
angles. The class contains every $S_\mu$ with constant $\mu$, the
optimally tuned $S_{\mu^\star}$ among them, and the bound is
attained by the parameter-free choice $c(v)=\mu_v$, that is, by
$C_T$.
\end{corollary}

\begin{proof}
If $T(v)=v$, both sides of the first inequality vanish. Otherwise
$\Phi_c(v)=S_{c(v)}(v)$, and Corollary~\ref{lem:mu-optimal}
applies. The second inequality follows by evaluating the first on
the worst-case ray~\eqref{eq:worst-ray}, on which
$\norm{C_T(v^\star)-\bar v}=\rho_V\norm{v^\star-\bar v}$.
\end{proof}

On the worst-case ray the adaptive step selects the optimal
constant without being told the angles: with
$w^\star=v^\star-\bar v$, the computation at the end of the proof
of Lemma~\ref{lem:kantorovich} gives $\scal{w^\star}{Mw^\star}=2ab$
and $\scal{w^\star}{M^2w^\star}=ab(a+b)$, so
$\mu_{v^\star}=2/(a+b)=\mu^\star$ and
$C_T(v^\star)=S_{\mu^\star}(v^\star)$. Elsewhere on $V$, $C_T$ is
strictly closer to $\bar v$ than $S_{\mu^\star}$ whenever
$\mu_v\ne\mu^\star$, by~\eqref{eq:CT-vs-S-mu}.

\section{The strict inequality and special configurations}\label{sec:strict-and-special}

The new rate $\rho_V$ is bounded above by $c_F^2$ in the rate
hierarchy of~\eqref{eq:rate-hierarchy}. The next proposition
shows that the gap is strictly positive in all but one boundary
configuration, namely $\theta_F=\theta_p=\pi/2$.

\begin{proposition}\label{prop:strict}
With $a=\sin^2\theta_F\in(0,1]$ and $b=\sin^2\theta_p\in[a,1]$,
\begin{equation}\label{eq:cF-rhoV-gap}
   c_F^2-\rho_V  =  \frac{a(2-a-b)}{a+b}  \ge  0,
\end{equation}
with equality iff $a=b=1$ (i.e., $\theta_F=\theta_p=\pi/2$). In
particular, $\rho_V<c_F^2$ whenever $\theta_F<\pi/2$.
\end{proposition}

\begin{proof}
A direct calculation gives
\[
   c_F^2-\rho_V
   =(1-a)-\frac{b-a}{a+b}
   =\frac{(1-a)(a+b)-(b-a)}{a+b}
   =\frac{2a-a^2-ab}{a+b}
   =\frac{a(2-a-b)}{a+b},
\]
which is~\eqref{eq:cF-rhoV-gap}. Since $a>0$ and $a+b>0$, the
sign of $c_F^2-\rho_V$ is the sign of $2-a-b$, and $a+b\le 2$
(with equality iff $a=b=1$) gives the inequality and its
equality case.
\end{proof}

\begin{remark}[The boundary case $\theta_F=\pi/2$]\label{rem:cF-zero}
The extreme value $c_F=0$ corresponds to $\theta_F=\pi/2$, which
forces $\theta_p=\pi/2$ as well (since
$\theta_F\le\theta_p\le\pi/2$). In this configuration $a=b=1$,
$\rho_V=0$, and the rate hierarchy~\eqref{eq:rate-hierarchy}
collapses to $0\le 0\le 0\le 0<1$ with every inequality except
the last non-strict. The conclusion matches
Corollary~\ref{cor:one-step}: CRM started from $V$ converges in
a single step, $C_T(P_Vx_0)=\bar x_0$ for every $x_0\in\RR^n$.
The spectral proof of Theorem~\ref{thm:rate} goes through
verbatim: the Kantorovich bound $4ab/(a+b)^2=1$ forces the
residual~\eqref{eq:residual} to vanish. The strict statement
$\rho_V<c_F^2$ in Theorem~\ref{thm:main} is correctly
conditional on $\theta_F<\pi/2$.
\end{remark}

\begin{proof}[{\bf Proof of Theorem~\ref{thm:main}}]
The pointwise contraction~\eqref{eq:main-rate} is
Theorem~\ref{thm:rate} applied to $v$ with
$\bar v=P_{U\cap V}(v)$, and $\rho_V\le c_F^2$, strict for
$\theta_F<\pi/2$, is Proposition~\ref{prop:strict}. For the
equality set~\eqref{eq:main-rate-sharp}, write
$w^\star\coloneqq\sin\theta_p\,f_{s+1}+\sin\theta_F\,f_p$;
Theorem~\ref{thm:rate} gives equality at $w^\star$. Every
$u\in U\cap V$ is fixed by $R_U$ and $R_V$, so
$C_T(x+u)=C_T(x)+u$, and $R_U,R_V$ are linear, so
$C_T(tx)=t\,C_T(x)$ for $t\in\RR$. Both sides
of~\eqref{eq:main-rate} scale in the same way, and equality at
$w^\star$ carries over to $u+tw^\star$ for every $u\in U\cap V$
and $t\in\RR$.

For the orbit statement, two facts are proved, not assumed.
First, $v_0=P_V(x_0)\in V$, and $v_k\in V$ implies
$v_{k+1}=C_T(v_k)\in V$ by Theorem~\ref{thm:linesearch};
moreover $U\cap V\subset V$ gives
$P_{U\cap V}=P_{U\cap V}\circ P_V$, so
$P_{U\cap V}(v_0)=P_{U\cap V}(x_0)=\bar x_0$. Second, the
increment $v_{k+1}-v_k=\mu_{v_k}(T(v_k)-v_k)=-\mu_{v_k}Mv_k$
from~\eqref{eq:linesearch} lies in the range of $M$, which is
contained in $V\cap(U\cap V)^\perp$ by Lemma~\ref{lem:spectrum};
the increment is therefore orthogonal to $U\cap V$, and
$P_{U\cap V}(v_{k+1})=P_{U\cap V}(v_k)$. The degenerate case
$T(v_k)=v_k$ gives $v_{k+1}=v_k$. Applying~\eqref{eq:main-rate}
at $v=v_k$ and iterating yields~\eqref{eq:main-rate-orbit}.
\end{proof}

The closed form
$\rho_V=(\sin^2\theta_p-\sin^2\theta_F)/(\sin^2\theta_p+\sin^2\theta_F)$
answers two further questions directly: which configurations of
$(U,V)$ make CRM converge in a single step from $V$, and how
large the rate can be once the Friedrichs angle is fixed. We
take them in order.

\subsection{One-step convergence}\label{sec:special}

One-step convergence is the degenerate end of the rate scale,
$\rho_V=0$. Since $\rho_V$ is a quotient, it vanishes exactly
when its numerator does, i.e., $\theta_F=\theta_p$; the corollary
below records this alongside the equivalent operator-theoretic
and geometric formulations.

\begin{corollary}[Characterization of one-step convergence]\label{cor:one-step}
The following are equivalent:
\begin{enumerate}[label={\rm(\alph*)},leftmargin=2.5em,topsep=2pt,itemsep=2pt]
\item\label{it:one-step-a} $\theta_F=\theta_p$;
\item\label{it:one-step-b} $\rho_V=0$;
\item\label{it:one-step-c} $C_T$ started from $V$ converges in
one step: $C_T(P_Vx_0)=\bar x_0$ for every $x_0\in\RR^n$;
\item\label{it:one-step-d} $M|_{V\cap(U\cap V)^\perp}$ is a
positive scalar multiple of the identity.
\end{enumerate}
A sufficient (but not necessary) geometric condition is
$\dim V-\dim(U\cap V)=1$.
\end{corollary}

\begin{proof}
(a)$\Leftrightarrow$(b) is immediate from
$\rho_V=(b-a)/(b+a)=0\Leftrightarrow a=b$. For
(a)$\Leftrightarrow$(d), the spectrum of
$M|_{V\cap(U\cap V)^\perp}$ is
$\{\sin^2\theta_{s+1},\ldots,\sin^2\theta_p\}$ by
Lemma~\ref{lem:spectrum}, and a self-adjoint operator is a
scalar multiple of the identity iff its spectrum is a singleton,
which here means $\theta_F=\theta_p$. The implication
(b)$\Rightarrow$(c) is immediate from Theorem~\ref{thm:rate}.
For (c)$\Rightarrow$(b), apply the equality case of
Theorem~\ref{thm:rate} at the worst-case ray $v^\star$: equality
$\norm{C_T(v^\star)-\bar v}=\rho_V\norm{v^\star-\bar v}$ with
$v^\star\neq\bar v$ forces $\rho_V=0$.

For the sufficient condition, $\dim V-\dim(U\cap V)=p-s=1$
implies the only positive principal angle is
$\theta_p=\theta_{s+1}=\theta_F$, so (a) applies. The converse
fails: $V\subset U^\perp$ with $\dim V\ge 2$ has $p-s\ge 2$ but
all positive principal angles equal $\pi/2$, so still
$\theta_F=\theta_p$.
\end{proof}

\begin{remark}[Comparison with $S_{\mu^\star}$]\label{rem:bbnpw-onestep}
\citep[Remark~3.10]{BBNPW2016} proves the parallel statement for
the linear $S_\mu$: it has rate $0$ if and only if
$\theta_F=\theta_p$ and
$\mu=\mu^\star=2/(\sin^2\theta_F+\sin^2\theta_p)$. Compared
with Corollary~\ref{cor:one-step}, this shows the practical
advantage of $C_T$ over $S_{\mu^\star}$: the same one-step
convergence is achieved without prior knowledge of the
principal angles, since $\mu_v$ in~\eqref{eq:linesearch} is
computed directly from $v$.
\end{remark}

The sufficient condition $\dim V-\dim(U\cap V)=1$ in
Corollary~\ref{cor:one-step} is satisfied in three classical
configurations, which we record as separate corollaries. The
first is the line case.

\begin{corollary}[$V$ is a line]\label{cor:V-line}
If $\dim V=1$ and $V\not\subset U$, then $\rho_V=0$ and CRM
converges in one step.
\end{corollary}

\begin{proof}
Here $p=\dim V=1$, so the principal angle vector is the
singleton $(\theta_1)$. Since $V\not\subset U$, we have
$V\cap U=\{0\}$ (else $V=V\cap U\subset U$), so $s=0$ and
$\theta_F=\theta_{s+1}=\theta_1=\theta_p$.
Corollary~\ref{cor:one-step} applies.
\end{proof}

The second is the hyperplane case, which recovers a result of
Behling, Bello-Cruz, and Santos~\citep[Theorem~6.1]{BBS2020}.

\begin{corollary}[Two distinct hyperplanes]\label{cor:hyperplanes}
If $\dim U=\dim V=n-1$ and $U\neq V$, then $\rho_V=0$ and CRM
converges in one step.
\end{corollary}

\begin{proof}
$U+V=\RR^n$ since $U\neq V$ are both of codimension $1$, so
$\dim(U\cap V)=\dim U+\dim V-\dim(U+V)=2(n-1)-n=n-2$. Hence
$p-s=(n-1)-(n-2)=1$, and Corollary~\ref{cor:one-step} applies.
\end{proof}

A different sufficient condition, not implied by
$\dim V-\dim(U\cap V)=1$, is that $V$ be entirely orthogonal
to $U$.

\begin{corollary}[$V$ is orthogonal to $U$]\label{cor:V-perp-U}
If $V\subset U^\perp$ (so in particular $U\cap V=\{0\}$), then
$\rho_V=0$ and CRM converges in one step. This holds regardless
of $\dim V$.
\end{corollary}

\begin{proof}
For every $v\in V\subset U^\perp$, $P_Uv=0$, so $T(v)=0$ and
$Mv=v$. Hence $M|_V=\Id_V$ with spectrum $\{1\}=\{\sin^2(\pi/2)\}$,
so by Lemma~\ref{lem:spectrum} all principal angles equal
$\pi/2$, giving $\theta_F=\theta_p=\pi/2$, and
Corollary~\ref{cor:one-step} applies.
\end{proof}

The instances listed for $S_{\mu^\star}$
in~\citep[Remark~3.10]{BBNPW2016}, two distinct lines through
the origin in $\RR^2$, a line and a hyperplane in $\RR^3$ with
the line not in the hyperplane, and two distinct hyperplanes in
$\RR^3$, all satisfy $p-s=1$ and are covered by
Corollaries~\ref{cor:V-line}--\ref{cor:hyperplanes};
Corollary~\ref{cor:V-perp-U} adds the orthogonal case not
captured by $p-s=1$.

\subsection{Worst-case rate for fixed Friedrichs angle}\label{ssec:worst-case}

The Friedrichs angle is the natural geometric parameter: it is
symmetric in $(U,V)$ and governs the rates of MAP, DRM, and the
original CRM analysis. For fixed $\theta_F$, the rate $\rho_V$
depends additionally on $\theta_p$, and the worst case is
achieved at the boundary $\theta_p=\pi/2$.

\begin{proposition}[Worst-case rate for fixed $\theta_F$]\label{prop:worst-case}
For fixed $\theta_F\in(0,\pi/2)$,
\begin{equation}\label{eq:worst-case}
   \sup_{\theta_p\in[\theta_F,\pi/2]}\rho_V
     =
   \frac{1-\sin^2\theta_F}{1+\sin^2\theta_F}
     =
   \frac{c_F^2}{2-c_F^2},
\end{equation}
attained at $\theta_p=\pi/2$ (which requires
$V\cap U^\perp\neq\{0\}$). The worst-case rate satisfies
\begin{equation}\label{eq:worst-bound}
   \frac{c_F^2}{2-c_F^2}  <  c_F^2
   \quad\text{whenever }\theta_F\in(0,\pi/2),
\end{equation}
with explicit gap
\begin{equation}\label{eq:gap-factored}
   c_F^2-\frac{c_F^2}{2-c_F^2}
     =  \frac{c_F^2(1-c_F^2)}{2-c_F^2}.
\end{equation}
\end{proposition}

\begin{proof}
Fix $a=\sin^2\theta_F\in(0,1)$ and view $\rho_V(b)=(b-a)/(b+a)$
as a function of $b\in[a,1]$. The derivative
$d\rho_V/db=2a/(a+b)^2>0$ shows $\rho_V$ is strictly increasing
in $b$, so the supremum is attained at $b=1$ with value
$(1-a)/(1+a)$. The substitution $\sin^2\theta_F=1-c_F^2$
rewrites this as $c_F^2/(2-c_F^2)$, giving the second form
in~\eqref{eq:worst-case}. The supremum point $b=1$ corresponds
to $\theta_p=\pi/2$, equivalently $V\cap U^\perp\neq\{0\}$
by~\eqref{eq:frame-extremes}.

For the strict gap to $c_F^2$, write $t=c_F^2\in(0,1)$:
\[
   c_F^2-\frac{c_F^2}{2-c_F^2}=t\Bigl(1-\frac{1}{2-t}\Bigr)=\frac{t(1-t)}{2-t}>0,
\]
which is~\eqref{eq:gap-factored}.
\end{proof}

\subsection{Faster contraction off the worst-case ray}\label{ssec:non-extremal}

The bound $\rho_V$ is attained only at points whose error has
components in both extreme eigenspaces of $M$, in the proportion
fixed by Corollary~\ref{prop:rho-sharp}. On orbits supported on a
narrower spectral range, $C_T$ contracts strictly faster, with
rate determined by the Kantorovich factor of that range.

\begin{proposition}[Local rate on a spectral interval]\label{prop:local-rate}
Let
$\lambda_{-},\lambda_{+}\in\{\sin^2\theta_{s+1},\ldots,\sin^2\theta_p\}$
with $\lambda_{-}\le\lambda_{+}$, and let $E$ be the sum of all
eigenspaces of $M|_V$ with eigenvalues in $[\lambda_-,\lambda_+]$.
For every $v\in\bar v+E\setminus\{\bar v\}$,
\begin{equation}\label{eq:local-rate}
   \norm{C_T(v)-\bar v}  \le  \rho(\lambda_-,\lambda_+)\,\norm{v-\bar v},
   \quad
   \rho(\lambda_-,\lambda_+)  \coloneqq  \frac{\lambda_+-\lambda_-}{\lambda_++\lambda_-},
\end{equation}
and the bound is sharp on the eigenplane $\Span\{f_i,f_j\}$ for
indices $i,j$ with $\sin^2\theta_i=\lambda_-$,
$\sin^2\theta_j=\lambda_+$. In particular,
$\rho(\lambda_-,\lambda_+)\le\rho_V$, with equality iff
$\lambda_-=\sin^2\theta_F$ and $\lambda_+=\sin^2\theta_p$.
\end{proposition}

\begin{proof}
Set $w=v-\bar v\in E$. Since $E$ is invariant under $M$, the
proof of Theorem~\ref{thm:rate} applies with $A=M|_E$, whose
spectrum is contained in $[\lambda_-,\lambda_+]$.
Lemma~\ref{lem:kantorovich} with $a=\lambda_-$, $b=\lambda_+$
gives $\norm{C_T(v)-\bar v}/\norm{w}\le\rho(\lambda_-,\lambda_+)$,
sharp at $\sqrt{\lambda_+}f_i+\sqrt{\lambda_-}f_j$. The inequality
$\rho(\lambda_-,\lambda_+)\le\rho_V$ follows from monotonicity:
increasing $\lambda_+$ or decreasing $\lambda_-$ widens the
interval and increases the Kantorovich gap.
\end{proof}

\begin{corollary}[Single eigendirection: one-step convergence]\label{cor:single-eigendir}
If
$v-\bar v\in\ker(M|_V-\lambda\,\Id)$ for some
$\lambda\in\{\sin^2\theta_{s+1},\ldots,\sin^2\theta_p\}$, then
$C_T(v)=\bar v$.
\end{corollary}

\begin{proof}
Apply Proposition~\ref{prop:local-rate} with
$\lambda_-=\lambda_+=\lambda$, giving rate
$\rho(\lambda,\lambda)=0$.
\end{proof}

The corollary explains a feature of $C_T$ that the worst-case
bound $\rho_V$ obscures: when $w=v-\bar v$ lies entirely in a
single eigenspace of $M$, the linesearch
formula~\eqref{eq:linesearch} selects $\mu_v=1/\lambda$ and the
iterate lands on $\bar v$ in one step.

\subsection{Summary of special cases}\label{ssec:summary-table}

Table~\ref{tab:special-cases} classifies the configurations by
convergence type: finite (one-step termination, $\rho_V=0$)
versus linear (rate $0<\rho_V<1$).

\begin{table}[htbp]
\centering
\renewcommand{\arraystretch}{1.15}
\begin{tabular}{lccl}
\toprule
Configuration & $\rho_V$ & Convergence & Reason\\
\midrule
$\theta_F=\theta_p$ (general) & $0$ & finite (1 step) & $a=b\Rightarrow(b-a)/(b+a)=0$\\
$\dim V-\dim(U\cap V)=1$ & $0$ & finite (1 step) & single positive principal angle\\
$V$ a line, $V\not\subset U$ & $0$ & finite (1 step) & $p=1$, $s=0\Rightarrow\theta_F=\theta_p$\\
$U$, $V$ hyperplanes, $U\neq V$ & $0$ & finite (1 step) & $p=n-1$, $s=n-2$\\
$V\subset U^\perp$ & $0$ & finite (1 step) & $\theta_F=\theta_p=\pi/2$\\
\midrule
$\theta_F<\theta_p<\pi/2$ & $<c_F^2/(2-c_F^2)$ & linear, rate $\rho_V$ & generic case\\
$\theta_p=\pi/2$, $\theta_F<\pi/2$ & $c_F^2/(2-c_F^2)$ & linear (worst rate) & worst case at fixed $\theta_F$\\
\bottomrule
\end{tabular}
\caption{Special configurations of $(U,V)$ and the corresponding
convergence behavior of $C_T$ on $V$. The first block lists
equivalent or sufficient conditions for finite convergence; the
second block records the regime of linear convergence at rate
$\rho_V$.}
\label{tab:special-cases}
\end{table}

The hierarchy $\rho_V\le c_F^2/(2-c_F^2)<c_F^2<c_F<1$
from~\eqref{eq:rate-hierarchy} holds whenever $\theta_F<\pi/2$
(Propositions~\ref{prop:strict} and~\ref{prop:worst-case}).

\section{Beyond single-step methods}\label{sec:improvement}

Corollaries~\ref{prop:rho-sharp} and~\ref{prop:single-step-opt}
settle the question for single-step methods along $T(v)-v$: the
rate $\rho_V$ is attained, and no choice of step improves on it.
A faster universal rate requires changing the structure of the
iteration. This section examines three such changes, each tied to
one feature of $C_T$: memory (Section~\ref{ssec:chebyshev}), the
reflections inside the circumcenter
(Section~\ref{ssec:over-reflection}), and adaptivity, which we
give up by fixing the coefficients of a linear combination of the
CRM vertices (Section~\ref{ssec:linear-competitor}).
Section~\ref{ssec:improvement-summary} collects the answers.

\subsection{Acceleration via memory: Chebyshev semi-iteration}\label{ssec:chebyshev}

Among single-step parameter-free methods, the rate $\rho_V$ is
optimal (Corollary~\ref{prop:single-step-opt}), and $C_T$,
$A_T$, $B_T$ all attain it
(Theorem~\ref{thm:three-agree-on-V}). A strictly faster rate
becomes available once we drop the single-step constraint and
let $v_{k+1}$ depend on the two preceding iterates $v_{k-1}$
and $v_k$. Chebyshev semi-iteration applied to $T$ on $V$ does
exactly this. The construction is classical
\citep[Chapter~5]{Hageman-Young}; we state it in our notation,
then formalize the strict improvement as
Theorem~\ref{thm:cheb-strict}.

\begin{proposition}[Chebyshev acceleration of MAP on $V$]\label{prop:chebyshev}
Let $w_0\in V\cap(U\cap V)^\perp$. Set $a=\sin^2\theta_F$,
$b=\sin^2\theta_p$, $d=(a+b)/2$, and
\[
   r\coloneqq\frac{a+b}{b-a},\quad
   \sigma\coloneqq\frac{\sqrt{b}+\sqrt{a}}{\sqrt{b}-\sqrt{a}},
\]
so that $\sigma+\sigma^{-1}=2r$. Define the three-term recursion
\begin{equation}\label{eq:chebyshev-recursion}
   w_1  =  w_0-d^{-1}M w_0,
   \quad
   w_{k+1}  =  \omega_{k+1}\bigl(w_k-d^{-1}Mw_k\bigr)+(1-\omega_{k+1})w_{k-1}
   \quad(k\ge 1),
\end{equation}
with $\omega_2=2r^2/(2r^2-1)$ and
$\omega_{k+1}=4r^2/(4r^2-\omega_k)$ for $k\ge 2$. Then
\begin{equation}\label{eq:chebyshev-rate}
   \norm{w_k}  \le  \frac{2}{\sigma^k+\sigma^{-k}}\,\norm{w_0}
     \le  2\sigma^{-k}\norm{w_0},
\end{equation}
and the asymptotic per-step rate is
\begin{equation}\label{eq:cheb-rate}
   \rho_{\mathrm{Cheb}}  =  \frac{1}{\sigma}
     =  \frac{\sqrt{b}-\sqrt{a}}{\sqrt{b}+\sqrt{a}}
     =  \frac{\sin\theta_p-\sin\theta_F}{\sin\theta_p+\sin\theta_F}.
\end{equation}
\end{proposition}

\begin{proof}
This is the Chebyshev semi-iterative
method~\citep[Chapter~5]{Hageman-Young} for the linear iteration
$w\mapsto(\Id-d^{-1}M)w$, whose iteration operator has spectrum
in $[-1/r,1/r]$ on $V\cap(U\cap V)^\perp$. By induction,
$w_k=p_k(M)w_0$ with $p_k(t)=T_k(s(t))/T_k(r)$, in which $T_k$ is
the Chebyshev polynomial of the first kind and
$s(t)\coloneqq(d-t)/\delta$, $\delta\coloneqq(b-a)/2$, maps
$[a,b]$ onto $[-1,1]$ with $s(0)=r$; the three-term recurrence of
$T_k$ evaluated at $r$ produces the factors $\omega_k$ of the
statement. Since $|T_k|\le1$ on $[-1,1]$ and
$T_k(r)=(\sigma^k+\sigma^{-k})/2$, expanding $w_0$ in the
eigenbasis of Lemma~\ref{lem:spectrum}
gives~\eqref{eq:chebyshev-rate}, and the per-step factor of the
bound tends to $\sigma^{-1}=\rho_{\mathrm{Cheb}}$.
\end{proof}

\begin{theorem}[Chebyshev gives a strict improvement on $\rho_V$]\label{thm:cheb-strict}
Whenever $\theta_F<\theta_p$ (equivalently $a<b$),
\begin{equation}\label{eq:cheb-comparison}
   \rho_{\mathrm{Cheb}}  <  \rho_V,
\end{equation}
with explicit ratio
\begin{equation}\label{eq:cheb-ratio}
   \frac{\rho_V}{\rho_{\mathrm{Cheb}}}
     =  \frac{(\sqrt{a}+\sqrt{b})^2}{a+b}
     =  1+\frac{2\sqrt{ab}}{a+b}  \in  (1,2],
\end{equation}
attaining the upper limit $2$ as $a/b\to 1^-$ (where both rates
tend to~$0$). Equivalently,
\begin{equation}\label{eq:cheb-ratio-bound}
   \norm{w_k}\le 2\rho_{\mathrm{Cheb}}^k\norm{w_0}
   \quad\text{with}\quad
   \rho_{\mathrm{Cheb}}\le\rho_V\bigl(1+2\sqrt{ab}/(a+b)\bigr)^{-1}.
\end{equation}
\end{theorem}

\begin{proof}
Factorize
$\rho_V=(b-a)/(b+a)=(\sqrt{b}-\sqrt{a})(\sqrt{b}+\sqrt{a})/(a+b)$.
Dividing by
$\rho_{\mathrm{Cheb}}=(\sqrt{b}-\sqrt{a})/(\sqrt{b}+\sqrt{a})$
gives
$\rho_V/\rho_{\mathrm{Cheb}}=(\sqrt{b}+\sqrt{a})^2/(a+b)=1+2\sqrt{ab}/(a+b)$.
AM--GM gives $2\sqrt{ab}/(a+b)\le 1$, with equality iff $a=b$,
so $1<\rho_V/\rho_{\mathrm{Cheb}}\le 2$. The
bound~\eqref{eq:cheb-ratio-bound} is~\eqref{eq:chebyshev-rate}
restated using $\rho_{\mathrm{Cheb}}=1/\sigma$.
\end{proof}

\begin{remark}[Trade-offs]\label{rem:chebyshev-tradeoffs}
The Chebyshev rate~\eqref{eq:cheb-rate} is strictly faster than
$\rho_V$, but the algorithm carries three drawbacks not present
in $C_T$: the parameters $d$, $\sigma$, $\omega_k$ require
explicit knowledge of $\theta_F$ and $\theta_p$; two previous
iterates must be stored at each step; and the recursion becomes
numerically delicate when $\kappa\coloneqq b/a$ is large. The
last point can be quantified: $r=(\kappa+1)/(\kappa-1)$, so the
multipliers in~\eqref{eq:chebyshev-recursion} approach the
boundary values $\omega_k\to 2-4/\sqrt{\kappa}+O(1/\kappa)$ and
$1-\omega_k\to-1+4/\sqrt{\kappa}$, and each update is a
near-cancellation of two terms of comparable size, sustained
over the order of $\sqrt{\kappa}$ iterations that the rate
$\rho_{\mathrm{Cheb}}=1-2/\sqrt{\kappa}+O(1/\kappa)$ requires;
see \citet{GutknechtRollin2002} for the finite-precision
analysis of the coupled implementation.
We view $C_T$ as the optimal parameter-free, memoryless method
matching $\rho_V$. Whether a Chebyshev-type acceleration of
$C_T$ itself, exploiting its nonlinearity, can produce a
strictly better rate than $\rho_{\mathrm{Cheb}}$ is an open
question.
\end{remark}

The Chebyshev recursion has another instructive interpretation:
its inner step is itself a relaxed alternating-projection step,
and the relaxation parameter matches the optimal one used by
$C_T$ on the worst-case ray.

\subsection{Relaxed reflections inside the circumcenter}\label{ssec:over-reflection}

Instead of adding memory, one can try to build relaxation into
the circumcenter itself. For $\lambda,\eta\in\RR$ write
\[
   R_U^{(\lambda)}\coloneqq(1-\lambda)\Id+\lambda R_U=\Id-2\lambda P_{U^\perp},
   \qquad
   R_V^{(\eta)}\coloneqq(1-\eta)\Id+\eta R_V=\Id-2\eta P_{V^\perp},
\]
so that $\lambda=\eta=1$ gives back $R_U$ and $R_V$. The
parameters $\lambda$ and $\eta$ appear only in this subsection;
in particular, they are unrelated to the relaxation $\mu$ of
$S_\mu$. Relaxing the reflections does not help, and the reason
is geometric: on $V$ the relaxed vertices span the same plane
$W_v$ as the original ones, and $C_T(v)$ is already the point of
that plane closest to $\bar v$.

\begin{proposition}[Relaxed reflections cannot beat $C_T$]\label{prop:over-reflection-double}
Let $\lambda,\eta\in\RR$ with $\lambda\eta\ne0$, and set
\[
   C_T^{(\lambda,\eta)}(v)  \coloneqq  \circum\bigl\{v,\,R_U^{(\lambda)}(v),\,R_V^{(\eta)}R_U^{(\lambda)}(v)\bigr\}.
\]
For every $v\in V$, the three relaxed vertices have affine hull
$W_v$, and
\begin{equation}\label{eq:overrefl-pointwise}
   \norm{C_T^{(\lambda,\eta)}(v)-\bar v}  \ge  \norm{C_T(v)-\bar v},
\end{equation}
with equality if and only if $C_T^{(\lambda,\eta)}(v)=C_T(v)$.
Explicitly, for every $v\in V$,
\begin{equation}\label{eq:double-overrefl-formula}
   C_T^{(\lambda,\eta)}(v)  =  v-\lambda\bigl[2(1-\eta)+\mu_v(2\eta-1)\bigr]Mv+2\lambda(1-\eta)\,P_{V^\perp}P_Uv,
\end{equation}
with $\mu_v=\scal{v}{Mv}/\norm{Mv}^2$ and the convention
$\mu_vMv=0$ when $Mv=0$. Consequently, the worst-case contraction
factor of $C_T^{(\lambda,\eta)}$ on $V$ is at least $\rho_V$. If
$\theta_F<\pi/2$, the contraction ratio of $C_T^{(\lambda,\eta)}$
on the worst-case ray~\eqref{eq:worst-ray} exceeds $\rho_V$ unless
$(\lambda,\eta)=(1,1)$. If $\theta_F=\pi/2$, then
$C_T^{(\lambda,\eta)}(v)-\bar v=(1-\lambda)(v-\bar v)$ for every
$v\in V$.
\end{proposition}

\begin{proof}
Put $y\coloneqq R_U^{(\lambda)}(v)=v-2\lambda P_{U^\perp}v$ and
$z\coloneqq R_V^{(\eta)}(y)=y-2\eta P_{V^\perp}y$. By
Lemma~\ref{lem:plane}, $P_{U^\perp}v=Mv-P_{V^\perp}P_Uv$, and
$P_{V^\perp}y=2\lambda P_{V^\perp}P_Uv$ because
$P_{V^\perp}v=P_{V^\perp}Mv=0$. Hence
\begin{equation}\label{eq:overrefl-vertices}
   y-v=-2\lambda\bigl(Mv-P_{V^\perp}P_Uv\bigr),\qquad
   z-v=-2\lambda\bigl(Mv-(1-2\eta)P_{V^\perp}P_Uv\bigr).
\end{equation}
The coefficient matrix of this pair with respect to
$(Mv,P_{V^\perp}P_Uv)$ has determinant $8\lambda^2\eta\ne0$, so
$\Span\{y-v,z-v\}=\Span\{Mv,P_{V^\perp}P_Uv\}$ and
$\aff\{v,y,z\}=W_v$ by Lemma~\ref{lem:plane}. The circumcenter of
Definition~\ref{def:circum} lies in the affine hull of its
points, so $C_T^{(\lambda,\eta)}(v)\in W_v$, while
$C_T(v)=P_{W_v}(\bar v)$ is the unique point of $W_v$ closest to
$\bar v$ (Lemma~\ref{lem:BBS}). This
proves~\eqref{eq:overrefl-pointwise} and its equality case.

For~\eqref{eq:double-overrefl-formula}, assume first that $Mv$
and $P_{V^\perp}P_Uv$ are both nonzero, so that $v,y,z$ are
affinely independent and $C\coloneqq C_T^{(\lambda,\eta)}(v)$ is
the point of $W_v$ equidistant from them. The relaxed reflections
are not isometries, so $\bar v$ is in general not equidistant
from $v,y,z$, and Lemma~\ref{lem:BBS} does not locate $C$; the
equidistance conditions $2\scal{C-v}{y-v}=\norm{y-v}^2$ and
$2\scal{C-v}{z-v}=\norm{z-v}^2$ do. Write
$C-v=c_1Mv+c_2P_{V^\perp}P_Uv$, $m\coloneqq\norm{Mv}^2$, and
$q\coloneqq\norm{P_{V^\perp}P_Uv}^2$.
By~\eqref{eq:overrefl-vertices} and orthogonality, the two
conditions read
\begin{align}
   -c_1m+c_2q&=\lambda(m+q),\label{eq:overrefl-system-1}\\
   -c_1m+(1-2\eta)\,c_2q&=\lambda\bigl(m+(1-2\eta)^2q\bigr).\label{eq:overrefl-system-2}
\end{align}
Subtracting and dividing by $2\eta q\ne0$ gives
$c_2=2\lambda(1-\eta)$. Substituting
in~\eqref{eq:overrefl-system-1} and using $\mu_vm=m+q$, which
is~\eqref{eq:pythag-V}, gives
$c_1=-\lambda[2(1-\eta)+\mu_v(2\eta-1)]$, and this
is~\eqref{eq:double-overrefl-formula}. If $Mv\ne0$ and
$P_{V^\perp}P_Uv=0$, then $\mu_v=1$ by~\eqref{eq:pythag-V}, the
vertex $y=v-2\lambda Mv\in V$ is fixed by $R_V^{(\eta)}$, so
$z=y$, and Definition~\ref{def:circum} returns the midpoint
$v-\lambda Mv$; the bracket in~\eqref{eq:double-overrefl-formula}
equals $2(1-\eta)+(2\eta-1)=1$, so the formula returns the same
point. If $Mv=0$, then $v\in U\cap V$ by
Theorem~\ref{thm:linesearch}, the three vertices coincide with
$v$, and both sides equal $v$.

For the last assertions, evaluate~\eqref{eq:overrefl-pointwise}
at $v^\star$, at which
$\norm{C_T(v^\star)-\bar v}=\rho_V\norm{v^\star-\bar v}$
(Theorem~\ref{thm:rate}). With $w^\star=v^\star-\bar v$, the
eigenrelations of Lemma~\ref{lem:spectrum} give
$\norm{Mv^\star}^2=ab(a+b)$ and
$\norm{P_{V^\perp}P_Uv^\star}^2=\norm{P_Uw^\star}^2-\norm{P_VP_Uw^\star}^2=ab(2-a-b)$,
both positive when $\theta_F<\pi/2$. The points
$C_T^{(\lambda,\eta)}(v^\star)$ and
$C_T(v^\star)=v^\star-\mu_{v^\star}Mv^\star$ then coincide only if
their coefficients on the orthogonal pair
$Mv^\star,P_{V^\perp}P_Uv^\star$ agree, that is, only if
$2\lambda(1-\eta)=0$ and
$\lambda[2(1-\eta)+\mu_{v^\star}(2\eta-1)]=\mu_{v^\star}$, which
forces $\eta=1$ and then $\lambda=1$. For any other pair the
inequality~\eqref{eq:overrefl-pointwise} is strict at $v^\star$.
If $\theta_F=\pi/2$, every positive principal angle equals
$\pi/2$, so $P_Uv=\bar v$ for $v\in V$; then
$P_{V^\perp}P_Uv=0$, $Mv=v-\bar v$, $\mu_v=1$,
and~\eqref{eq:double-overrefl-formula} reduces to
$v-\lambda(v-\bar v)$.
\end{proof}

\begin{remark}[Relaxing only $R_U$ rescales the CRM step]\label{rem:over-reflection-CRM}
For $\eta=1$, formula~\eqref{eq:double-overrefl-formula} reads
\begin{equation}\label{eq:overrefl-CRM-formula}
   C_T^{(\lambda,1)}(v)=v-\lambda\mu_vMv=v+\lambda\bigl(C_T(v)-v\bigr)
   \quad(v\in V),
\end{equation}
so relaxing $R_U$ alone rescales the CRM step by the factor
$\lambda$ along the same direction. For $\lambda\in[0,2]$ and
$v\in V\setminus(U\cap V)$, the expansion leading
to~\eqref{eq:residual} gives, with $w=v-\bar v$,
\[
   \frac{\norm{C_T^{(\lambda,1)}(v)-\bar v}^2}{\norm{v-\bar v}^2}
   =1-\lambda(2-\lambda)\,\frac{\scal{w}{Mw}^2}{\norm{w}^2\,\scal{w}{M^2w}},
\]
and Lemma~\ref{lem:kantorovich} yields the worst-case contraction
\[
   \rho_V^{(\lambda)}  =  \sqrt{\,1-\frac{4\lambda(2-\lambda)\,ab}{(a+b)^2}\,}  \ge  \rho_V,
\]
with equality if and only if $\lambda=1$, attained on the
worst-case ray~\eqref{eq:worst-ray}.
\end{remark}

The proof uses nothing about the relaxed reflections beyond the
plane their vertices span. Any rule that returns a point of
$W_v$, whether a circumcenter, an affine combination of the
vertices, or a nonlinear selection among them, stays at least as
far from $\bar v$ as $C_T(v)$, by Lemma~\ref{lem:BBS}. This
extends Corollary~\ref{prop:single-step-opt} from the line
$v+\RR\,Mv$ to the whole plane.

\subsection{A linear competitor with constant coefficients}\label{ssec:linear-competitor}

The last change keeps the three CRM vertices and a single step,
but gives up adaptivity: the circumcenter is replaced by a fixed
affine combination of $x$, $R_Ux$, and $R_VR_Ux$,
\begin{equation}\label{eq:CDR-parametric-main}
   C_{T,\gamma,\beta}\coloneqq(1-\gamma-\beta)\,\Id+\beta\,R_VR_U+\gamma\,R_U,
   \qquad\gamma,\beta\in\RR.
\end{equation}
The family contains the Douglas--Rachford operator, at
$(\gamma,\beta)=(0,\tfrac12)$, and is linear for each fixed pair.
Its analysis is elementary but long, and it is given in full in
Online Resource~1 (Section~S1); we record the facts used
later. For $\gamma,\beta>0$ with $\gamma+\beta<1$, the iterates
$C_{T,\gamma,\beta}^k(x)$ converge to $P_{U\cap V}(x)$ for every
$x\in\RR^n$, with asymptotic rate $\Gamma(\gamma,\beta)$ equal to
the largest modulus of an eigenvalue of $C_{T,\gamma,\beta}$
other than $1$, which is available in closed form in terms of the
principal angles (Online Resource~1, Theorem~S1.4
and Corollary~S1.5). We write
$\Gamma^\star(\theta_F,\theta_p)\coloneqq\inf_{\gamma,\beta}\Gamma(\gamma,\beta)$
for the best rate, attained at a pair $(\gamma^\star,\beta^\star)$
that depends on both angles. When the subspaces $U\cap V^\perp$,
$V\cap U^\perp$, and $(U+V)^\perp$ are all nontrivial, the
eigenvalues $1-2\beta$, $1-2\gamma-2\beta$, and $1-2\gamma$ are
present, and
\begin{equation}\label{eq:third-floor}
   \Gamma(\gamma,\beta)\ge\max\bigl\{|1-2\gamma|,\,|1-2\beta|,\,|1-2\gamma-2\beta|\bigr\}\ge\tfrac13,
\end{equation}
because $|1-2\gamma|<\tfrac13$ and $|1-2\beta|<\tfrac13$ force
$1-2\gamma-2\beta=(1-2\gamma)+(1-2\beta)-1<-\tfrac13$. The floor
binds for large $\theta_F$ (Figure~\ref{fig:rates-comparison});
in the ill-conditioned regime $a\ll b$, on the other hand,
$\Gamma^\star$ can be much smaller than $\rho_V$, so neither rate
dominates the other (Online Resource~1,
Remark~S1.7). Composing with $P_V$ removes the
parasitic eigenvalues: on $V$, $C_{T,\gamma,\beta}P_V$ coincides
with $S_{2(\gamma+\beta)}$, and its best rate over
$(\gamma,\beta)$ is exactly $\rho_V$, attained when
$\gamma+\beta=1/(a+b)=\mu^\star/2$ (Online Resource~1,
Theorem~S1.11). Constant coefficients can
therefore match the adaptive $C_T$ only after a projection onto
$V$ and only with both angles known; $C_T$ needs neither.

\subsection{Summary}\label{ssec:improvement-summary}

The contents of this section are summarized in
Table~\ref{tab:improvement}.

\begin{table}[htbp]
\centering
\renewcommand{\arraystretch}{1.2}
\begin{tabular}{lcl}
\toprule
Question & Answer & Reference\\
\midrule
Is $\rho_V$ sharp on $V$? & yes, maximizers identified & Cor.~\ref{prop:rho-sharp}\\
Faster rate on a strict spectral subset? & yes, $\rho(\lambda_-,\lambda_+)$ & Prop.~\ref{prop:local-rate}\\
Single eigendirection? & rate $0$ (one step) & Cor.~\ref{cor:single-eigendir}\\
Better step along $T(v)-v$? & no, $\rho_V$ is optimal & Cor.~\ref{prop:single-step-opt}\\
Relax $R_U$ only? & no, rate $\ge\rho_V$ & Rem.~\ref{rem:over-reflection-CRM}\\
Relax both reflections? & no, never closer than $C_T$ & Prop.~\ref{prop:over-reflection-double}\\
Fixed affine combination on $\RR^n$? & floor $\Gamma^\star\ge 1/3$ & \eqref{eq:third-floor}; OR1\\
Same, composed with $P_V$? & rate $\rho_V$ when tuned & Sect.~\ref{ssec:linear-competitor}; OR1\\
Two-step memory? & yes, $\rho_{\mathrm{Cheb}}<\rho_V$ & Thm.~\ref{thm:cheb-strict}\\
\bottomrule
\end{tabular}
\caption{Summary of the improvability results of
Sections~\ref{sec:linesearch}--\ref{sec:improvement}; OR1 denotes
Online Resource~1.}
\label{tab:improvement}
\end{table}

In short, $C_T$ is optimal among single-step methods along
$T(v)-v$, and relaxing the reflections inside the circumcenter
cannot improve on it (Proposition~\ref{prop:over-reflection-double}).
Strictly better rates require two-step memory
($\rho_{\mathrm{Cheb}}$, with both angles known), a special
starting point (Proposition~\ref{prop:local-rate}), or, on
$\RR^n$ and in the ill-conditioned regime only, the tuned linear
family of Section~\ref{ssec:linear-competitor}, which is floored
at $1/3$ and matches $\rho_V$ once composed with $P_V$.

\section{Starting outside $V$}\label{sec:outside-V}

Theorem~\ref{thm:main} requires the iteration to be initialized
in $V$. To handle a general starting point $x_0\in\RR^n$, we
examine four strategies.

The first three are based on the geometric $C_T$ and differ in
their warm start. Strategy~A projects onto $V$ before iterating:
$v_0=P_V(x_0)$, $v_{k+1}=C_T(v_k)$. Strategy~B iterates $C_T$
directly from $x_0$: $x_{k+1}=C_T(x_k)$. Strategy~C inserts a
single CRM step before the projection:
$\tilde v_0=P_V(C_T(x_0))$, $\tilde v_{k+1}=C_T(\tilde v_k)$.
After the warm start, A and C produce orbits in $V$ and contract
at rate $\rho_V$ by Theorem~\ref{thm:main}, whereas B remains
off $V$ and contracts at rate $c_F$ from~\citep{BBS2018}.

The fourth strategy abandons the geometric $C_T$ for the
parametric linear family $C_{T,\gamma,\beta}$ of
Section~\ref{ssec:linear-competitor}. This removes the
$V$-asymmetry, but at the cost of two scalar parameters that
depend on the principal angles. Strategies~A and~B are analyzed
below; Strategies~C and~D are analyzed in Online Resource~1
(Section~S2).

\subsection{The three linesearch maps off $V$}\label{ssec:off-V}

By Theorem~\ref{thm:three-agree-on-V}, $A_T$, $B_T$, and $C_T$
coincide on $V$; off $V$ they are different operators. The map
$A_T$ relaxes between $x$ and $P_VP_Ux$ with no constraint on the
result, so its orbit is generically not in $V$, and $\rho_V$ is
reached only asymptotically. The map $B_T$ composes the same
step with $P_V$, so $B_T(x)\in V$ for every $x$; from the second
iterate on it acts as the on-$V$ linesearch,
and~\citep[Theorem~4.5]{BBNPW2016} gives, in our order,
\begin{equation}\label{eq:BT-rate-our-order}
   \norm{B_T^{k+1}(x)-\bar x}\le\rho_V^k\,c_F^2\,\norm{x-\bar x}\quad(x\in\RR^n),
\end{equation}
sharpening to $\rho_V^{k+1}\norm{x-\bar x}$ for $x\in V$. The
circumcenter $C_T$ contracts on $\RR^n$ at the rate $c_F$
of~\citep{BBS2018}, and attains $\rho_V$ from the first iterate
when started in $V$. Table~\ref{tab:three-maps} summarizes. A
similar coincidence underlies the AAMR/GAP analysis of
Arag\'on Artacho and Campoy~\citep[Theorem~4.1]{AragonCampoy2019},
in which the shadow sequence of AAMR matches the iterates of
generalized alternating projections at $\alpha_1=\alpha_2=2\beta$.
There the shared rate $(1-\sin\theta_F)/(1+\sin\theta_F)$ depends
on $\theta_F$ alone and requires tuning; the rate shared by
$A_T$, $B_T$, $C_T$ uses no parameter and depends on both angles.

\begin{table}[h]
\centering
\begin{tabular}{lccc}
\toprule
                          & $A_T$         & $B_T$         & $C_T$ \\
\midrule
Construction              & linesearch   & linesearch + $P_V$ & circumcenter \\
Operators used            & $P_U,P_V$     & $P_U,P_V$     & $R_U,R_V$ \\
Iterate in $V$ if $x\notin V$ & no        & after $1$ step & no \\
Rate on $V$               & $\rho_V$      & $\rho_V$      & $\rho_V$ \\
Rate from $\RR^n$         & $\rho_V$ asymptotic & $\rho_V$ from $2$nd iter & $c_F$ uniformly \\
\bottomrule
\end{tabular}
\caption{The three linesearch maps in CRM's order. On $V$ they
coincide (Theorem~\ref{thm:three-agree-on-V}); off $V$ they
differ. The `from $\RR^n$' entries for $C_T$ are proved in Proposition~\ref{prop:strategy-B}.}
\label{tab:three-maps}
\end{table}

\subsection{Strategy A: project first}\label{ssec:strategy-A}

The most direct way to handle a starting point outside $V$ is
to project it onto $V$ once and then apply the in-$V$ theory
unchanged. Because $P_V$ is nonexpansive and fixes $\bar x_0$,
the warm-start projection cannot increase the distance to the
solution, and Theorem~\ref{thm:main} governs the resulting
orbit.

\begin{proposition}[Strategy A]\label{prop:strategy-A}
For every $x_0\in\RR^n$ and every $k\in\NN$, the iterates
$v_0=P_V(x_0)$, $v_{k+1}=C_T(v_k)$ satisfy
\begin{equation}\label{eq:strategy-A}
   \norm{v_k-\bar x_0}  \le  \rho_V^k\,\norm{P_V(x_0)-\bar x_0}
     \le  \rho_V^k\,\norm{x_0-\bar x_0}.
\end{equation}
\end{proposition}

\begin{proof}
By Theorem~\ref{thm:main},
$\norm{v_k-\bar x_0}\le\rho_V^k\norm{v_0-\bar x_0}$. For the
second inequality, $\bar x_0\in U\cap V\subset V$, so
$\bar x_0=P_V(\bar x_0)$, and Pythagoras gives
\[
   \norm{x_0-\bar x_0}^2
     =  \norm{x_0-P_V x_0}^2+\norm{P_Vx_0-\bar x_0}^2
     \ge  \norm{P_Vx_0-\bar x_0}^2.
\]
\end{proof}

The first inequality in~\eqref{eq:strategy-A} is sharp, but the
contraction
$\norm{P_V(x_0)-\bar x_0}/\norm{x_0-\bar x_0}\in[0,1]$ admits no
uniform bound below~$1$: it can equal~$1$ when $x_0\in V$ (no
contraction) and can be arbitrarily small when $x_0$ is mostly
orthogonal to $V$.

\subsection{Strategy B: direct iteration without projection}\label{ssec:strategy-B}

Strategy~B avoids the auxiliary projection altogether. It is
correspondingly harder to analyze, since $C_T(x_0)$ is generally
not in $V$ and $C_T^k(x_0)$ remains off $V$ for every $k$.

\begin{proposition}[Strategy B]\label{prop:strategy-B}
For every $x_0\in\RR^n$ and every $k\in\NN$,
\begin{equation}\label{eq:strategy-B}
   \norm{C_T^k(x_0)-\bar x_0}  \le  c_F^k\,\norm{x_0-\bar x_0}.
\end{equation}
\end{proposition}

\begin{proof}
The original CRM analysis~\citep[Theorem~1]{BBS2018} establishes
the single-step contraction
$\norm{C_T(x)-P_{U\cap V}(x)}\le c_F\norm{x-P_{U\cap V}(x)}$ for
every $x\in\RR^n$. To iterate, we need the
projection-invariance property
\begin{equation}\label{eq:projection-invariance}
   P_{U\cap V}(C_T(x))=P_{U\cap V}(x)\quad(x\in\RR^n).
\end{equation}
The displacement $C_T(x)-x$ lies in $\Span\{s_U,s_V\}$, so it
suffices to show $s_U,s_V\perp U\cap V$. For any
$u\in U\cap V$, $P_Uu=P_Vu=u$, and self-adjointness of $P_U$
gives $\scal{s_U}{u}=2\scal{P_Ux-x}{u}=2\scal{x}{P_Uu-u}=0$. For
$s_V$, the adjoint identity $(R_VR_U)^*=R_UR_V$ together with
$R_Vu=R_Uu=u$ gives
$\scal{R_VR_Ux}{u}=\scal{x}{R_UR_Vu}=\scal{x}{u}$, hence
$\scal{s_V}{u}=0$. So
$C_T(x)-x\perp U\cap V$, which is~\eqref{eq:projection-invariance}.

Granting~\eqref{eq:projection-invariance}, the bound iterates:
with $\bar x_k\coloneqq P_{U\cap V}(C_T^k(x_0))=\bar x_0$ for
every $k$,
\[
   \norm{C_T^{k+1}(x_0)-\bar x_0}=\norm{C_T(C_T^k(x_0))-\bar x_k}\le c_F\norm{C_T^k(x_0)-\bar x_0}.
\]
\end{proof}

The bound~\eqref{eq:strategy-B} is the sharpest known uniform
rate for Strategy~B; the warm-start strategies A and~C recover
the on-$V$ rate $\rho_V$ at the cost of a single projection onto
$V$.

\begin{remark}[Empirical behavior of Strategy~B]\label{rem:strategy-B-empirical}
The experiments of Section~\ref{ssec:warm-start-numerics}
clarify the asymptotics of Strategy~B. Over $600$ random starts
$x_0\in\RR^n$, the empirical asymptotic per-step ratio of
$\norm{C_T^k(x_0)-\bar x_0}$ never exceeds the bound $c_F$
of~\eqref{eq:strategy-B} (sample mean $0.89\,c_F$), and it does
not approach $\rho_V$: its mean is $2.2\,\rho_V$, with fewer
than $1\%$ of the starts within $5\%$ of $\rho_V$. The reason is
structural: for $x_0\notin V$ the orbit stays off $V$
(Table~\ref{tab:three-maps}), so the in-$V$ contraction of
Theorem~\ref{thm:main} never takes over, and a uniform bound of
the form
$\norm{C_T^k(x_0)-\bar x_0}\le C(x_0)\,\rho_V^k\norm{x_0-\bar x_0}$
does not hold. Attaining $\rho_V$ from a general starting point
requires the warm start of Strategy~A or~C.
\end{remark}

The contrast between the uniform $c_F$ bound for Strategy~B and
the in-$V$ rate $\rho_V$ has a clean geometric explanation: the
two iterations move in different dimensional regimes.

\begin{remark}[Why a 1D linesearch beats a 2D circumcenter in the worst case]\label{rem:1D-vs-2D}
A natural objection: $C_T(x)$ minimizes the distance to
$\bar x_0$ over the generically two-dimensional
$W_x=\aff\{x,R_Ux,R_VR_Ux\}$, while $C_T(v)$ for $v\in V$
minimizes over the one-dimensional line $W_v\cap V$; optimizing
over a larger set gives a smaller one-step error, so why does
the iterated rate from $\RR^n$ ($c_F$) exceed the iterated rate
from $V$ ($\rho_V<c_F^2$)? Because per-step optimality is not
per-orbit optimality: a greedy step can place the iterate where
subsequent contraction is slow, for the same reason steepest
descent is asymptotically slower than conjugate gradients on a
quadratic with the same per-step cost. On $V$ the search
direction $T(v)-v=-M(v-\bar v)$ is a spectral image of the
error, and the linesearch step
$\mu_v=\scal{w}{Mw}/\scal{w}{M^2w}$ with $w=v-\bar v$ is the
Kantorovich-optimal relaxation on $w$;
Theorem~\ref{thm:main} exploits exactly this alignment. For
general $x\in\RR^n$ the plane $W_x$ is not aligned with the
eigenstructure of $M$, and the circumcenter minimizes a
Euclidean distance in $\RR^n$, not a spectral functional of $M$.
Finally, the two rates are suprema over different sets: the
worst-case adversary in $\RR^n$ is a higher-dimensional set that
exploits the extra freedom to misalign $W_x$ with $\bar x_0$, so
$\rho_V<c_F$ does not contradict the per-step optimality of the
$\RR^n$ circumcenter.
\end{remark}

Two further strategies are analyzed in Online Resource~1
(Section~S2): Strategy~C, one $C_T$ step
before the projection, which keeps the rate $\rho_V$ and gains
the prefactor $c_F$ at the cost of one extra circumcenter
evaluation; and Strategy~D, the tuned linear family
$C_{T,\gamma^\star,\beta^\star}$ iterated on $\RR^n$, which
avoids the projection but pays the
floor~\eqref{eq:third-floor}. Both enter the summary below.

\subsection{Summary of warm-start strategies}\label{ssec:warm-start-summary}

The four strategies are summarized in
Table~\ref{tab:strategies}.

\begin{table}[htbp]
\centering
\renewcommand{\arraystretch}{1.2}
\begin{tabular}{llll}
\toprule
Strategy & Method & Asymptotic rate & Parameters\\
\midrule
A & project, then iterate $C_T$ & $\rho_V$ & none\\
B & iterate $C_T$ from $x_0$ & $c_F$ uniform & none\\
C & one $C_T$ step, project, iterate $C_T$ & $\rho_V$ (with prefactor $c_F$) & none\\
D & iterate $C_{T,\gamma^\star,\beta^\star}$ from $x_0$ & $\Gamma^\star(\theta_F,\theta_p)$ & $\theta_F,\theta_p$\\
\bottomrule
\end{tabular}
\caption{Summary of the four warm-start strategies; A and B are
analyzed in Section~\ref{sec:outside-V}, C and D in Online
Resource~1.}
\label{tab:strategies}
\end{table}

Among the four strategies, A and C deliver the sharp rate
$\rho_V$ of Theorem~\ref{thm:main} from any starting point in
$\RR^n$ at the cost of a single projection; both are
parameter-free, and C trades one extra $C_T$ evaluation for the
prefactor $c_F$. Against B the comparison is a worst-case one:
Strategy~A guarantees $\rho_V<c_F^2$ from step one, attained on
the worst-case ray, while B carries only the uniform bound
$c_F$, and empirically stays there: the asymptotic per-step
ratios have mean $0.89\,c_F$ and essentially never approach
$\rho_V$ (Remark~\ref{rem:strategy-B-empirical}), so the sharp
in-$V$ rate is not obtained by skipping the projection.
Omitting the warm start costs roughly the factor $c_F/\rho_V$ in
the rate. When $P_V$ is
a precomputed matrix--vector product the warm start is
essentially free, and even when $P_V$ costs as much as a $C_T$
step, the single projection pays for itself on all but the
shortest runs. The choice of starting subspace, $V$ or $U$, does
not affect the rate, since the principal angles are symmetric in
the pair and the analysis applies under the swap
$R_VR_U\leftrightarrow R_UR_V$; it should be driven by whichever
projection is cheaper, typically the lower-dimensional subspace.
In short: whenever the sharp rate matters, project once onto the
cheaper of $\{U,V\}$, then iterate $C_T$; the linesearch
identity of Theorem~\ref{thm:linesearch}, the Kantorovich-tight
bound of Theorem~\ref{thm:main}, and the finite-convergence
corollaries of Section~\ref{sec:special} all require starting on
the subspace.

\section{Analytic comparisons, a sharp instance, and numerical verification}\label{sec:numerics}

This section illustrates the rate hierarchy from
Theorem~\ref{thm:main} and Proposition~\ref{prop:worst-case}
and verifies the bound $\rho_V$ numerically. The rate curves
of Sections~\ref{ssec:rate-curves} and~\ref{ssec:cdr-comparison}
are plots of closed-form rate functions, including the
additional rates $\rho_{\mathrm{Cheb}}$ and $\Gamma^\star$ from
Theorem~\ref{thm:cheb-strict} and
Section~\ref{ssec:linear-competitor}, as
continuous functions of the Friedrichs angle.
Section~\ref{ssec:sharp-instance} gives an explicit sharp
instance in $\RR^4$ with a closed-form orbit. The verification
in Section~\ref{ssec:verification} evaluates the
circumcentered--reflection operator itself, in finite-precision
arithmetic, and recovers $\rho_V$ to unit roundoff.

\subsection{Rate curves as functions of $\theta_F$}\label{ssec:rate-curves}

Figure~\ref{fig:rates-theoretical} plots four contraction rates
as continuous functions of the Friedrichs angle
$\theta_F\in(0,\pi/2)$: $c_F=\cos\theta_F$, the rate shared by
the Douglas--Rachford method~\citep{BCNPW2014} and the original
CRM analysis~\citep{BBS2018}; $c_F^2$, the asymptotic MAP
rate~\citep[Theorem~9.31]{Deutsch}; $c_F^2/(2-c_F^2)$, the
worst-case CRM rate on $V$ at fixed $\theta_F$ from
Proposition~\ref{prop:worst-case}; and
$\rho_V(\theta_F,\theta_p)$ from Theorem~\ref{thm:main} for
three representative values of $\theta_p$. The plot exhibits the
hierarchy $\rho_V\le c_F^2/(2-c_F^2)<c_F^2<c_F$ predicted
by~\eqref{eq:rate-hierarchy} and
Proposition~\ref{prop:worst-case}.

\begin{figure}[h]
\centering
\includegraphics[width=11cm]{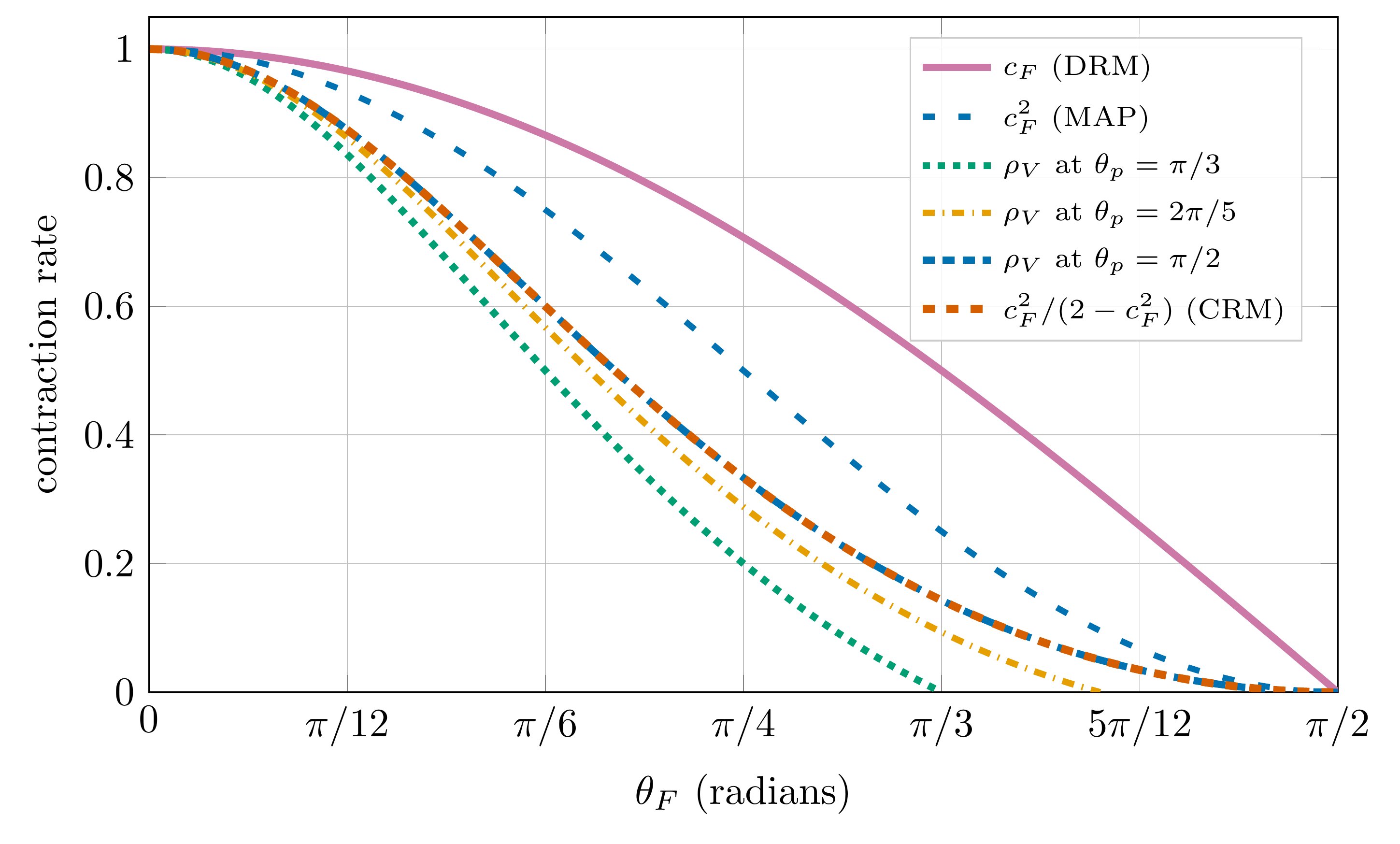}
\caption{Theoretical contraction rates as functions of $\theta_F$.
The CRM rate $\rho_V$ at $\theta_p=\pi/2$ coincides exactly with
the worst-case envelope $c_F^2/(2-c_F^2)$, since
$\rho_V|_{\theta_p=\pi/2}
=(1-\sin^2\theta_F)/(1+\sin^2\theta_F)
=c_F^2/(2-c_F^2)$.
For $\theta_p<\pi/2$ the rate is strictly smaller.
Moreover, $\rho_V$ vanishes at $\theta_F=\theta_p$
(one-step convergence, Corollary~\ref{cor:one-step}).}
\label{fig:rates-theoretical}
\end{figure}

Table~\ref{tab:method-comparison} in the introduction places
$C_T$ on $V$ alongside six other projection methods on two
subspaces, split between parameter-free worst-case rates and
rates achieved only with parameters tuned to the principal
angles. The two parameter-tuned rates
$(1-\sin\theta_F)/(1+\sin\theta_F)$ of GAP and AAMR are smaller
than $\rho_V$ at $\theta_p=\pi/2$ but larger than $\rho_V$ when
$\theta_p$ is close to $\theta_F$, since $\rho_V\to 0$ as
$\theta_p\to\theta_F$ (Corollary~\ref{cor:one-step}) while the
parameter-tuned rates depend only on $\theta_F$. The numerical
comparison in Section~\ref{ssec:aamr-comparison} maps out the
crossover in the $(\theta_F,\theta_p)$ plane.

\subsection{Comparison with $\rho_{\mathrm{Cheb}}$ and $\Gamma^\star$}\label{ssec:cdr-comparison}

Figure~\ref{fig:rates-comparison} adds two further rates from
Sections~\ref{ssec:chebyshev} and~\ref{ssec:linear-competitor} to the
picture, both at fixed $\theta_p=\pi/2$ (the worst case for
$\rho_V$): the Chebyshev semi-iteration rate
$\rho_{\mathrm{Cheb}}=(\sin\theta_p-\sin\theta_F)/(\sin\theta_p+\sin\theta_F)$
of~\eqref{eq:cheb-rate}, and the optimal rate
$\Gamma^\star(\theta_F,\theta_p)=\inf_{\gamma,\beta}\Gamma(\gamma,\beta)$
of the parametric family $C_{T,\gamma,\beta}$ characterized in Online Resource~1
(Theorem~S1.4 and Remark~S1.6).

\begin{figure}[h]
\centering
\includegraphics[width=11cm]{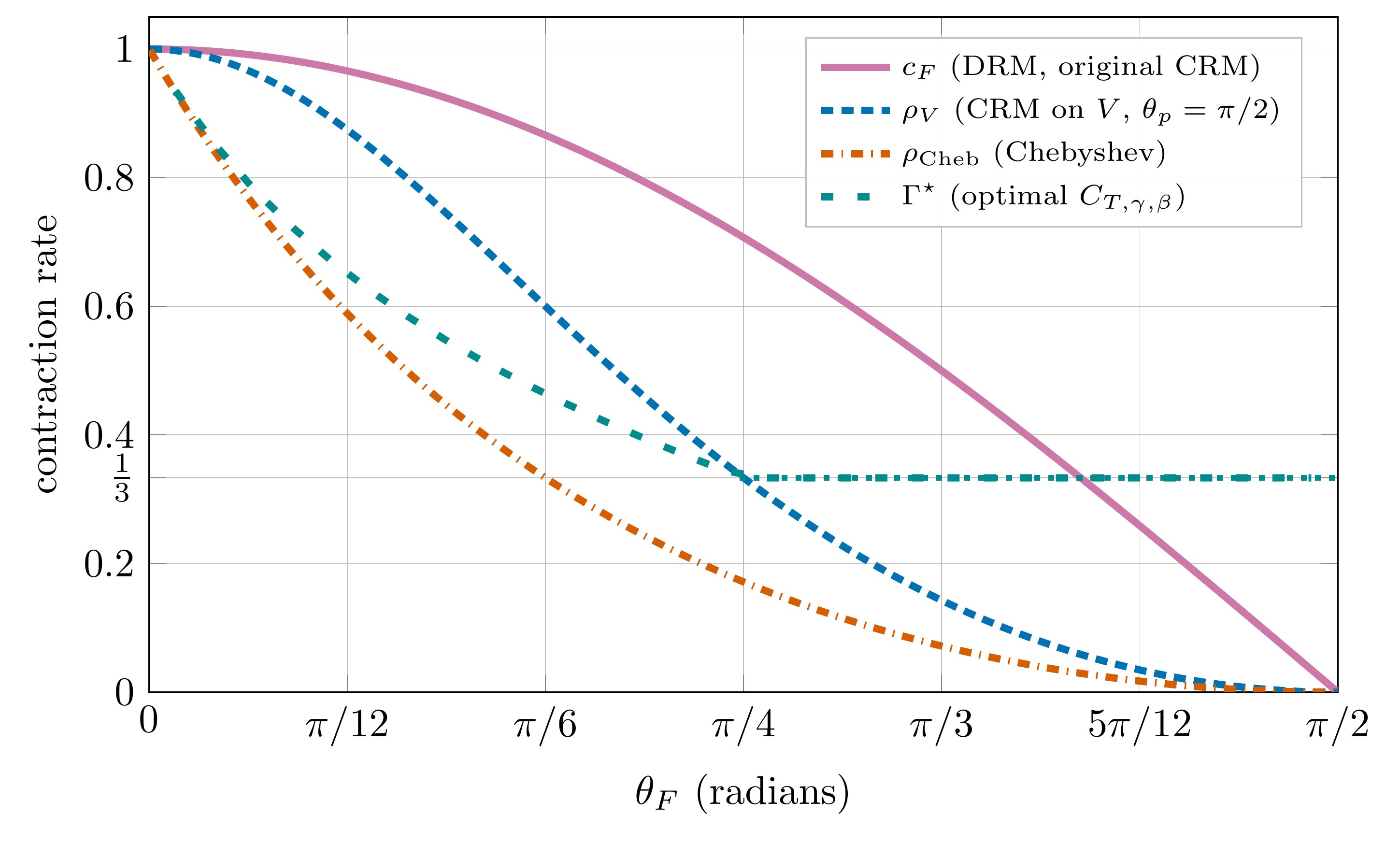}
\caption{Contraction rates at $\theta_p=\pi/2$.
The Chebyshev rate $\rho_{\mathrm{Cheb}}$ is uniformly below
$\rho_V$, with ratio
$\rho_V/\rho_{\mathrm{Cheb}}=(\sin\theta_p+\sin\theta_F)^2/(\sin^2\theta_p+\sin^2\theta_F)$
(Theorem~\ref{thm:cheb-strict}).
The optimal parametric CDR rate $\Gamma^\star$ is bounded below
by $1/3$, by~\eqref{eq:third-floor}; for
$\theta_F\gtrsim 0.78$ the floor binds and $\Gamma^\star=1/3$,
while $\rho_V$ continues decreasing to zero at
$\theta_F=\theta_p=\pi/2$.}
\label{fig:rates-comparison}
\end{figure}

\noindent
The horizontal floor $\Gamma^\star=1/3$ visible for $\theta_F$
large corresponds to the saturated optimum
$(\gamma^\star,\beta^\star)=(1/3,1/3)$, where the three
eigenvalues $|1-2\gamma|=|1-2\beta|=|1-2\gamma-2\beta|=1/3$
simultaneously bind (Online Resource~1,
Remark~S1.6). For each
$\theta_F\in(0,\pi/2)$ the value of
$\Gamma^\star(\theta_F,\pi/2)$ is the unique value at which the
active constraints of Online Resource~1,
Corollary~S1.5, balance. The crossover with $\rho_V$ occurs at the unique
$\theta_F^{\mathrm{cross}}\in(0,\pi/2)$ where
$\Gamma^\star(\theta_F^{\mathrm{cross}},\pi/2)=\rho_V(\theta_F^{\mathrm{cross}},\pi/2)$;
for $\theta_F$ above this value the geometric $C_T$ on $V$ has
the smaller rate, and for smaller $\theta_F$ the parametric
family does.

\subsection{A concrete sharp instance}\label{ssec:sharp-instance}

We close with an explicit instance of $(U,V)$ in $\RR^4$ for
which the bound of Theorem~\ref{thm:main} is realized at every
step of the orbit and every rate in the hierarchy admits a
closed form.

\begin{example}\label{ex:sharp-instance}
Let $V=\Span\{f_1,f_2\}$ in $\RR^4$ with $\{f_1,f_2,g_1,g_2\}$
an orthonormal basis, set
\[
   u_1  \coloneqq  \cos\tfrac{\pi}{6}\,f_1+\sin\tfrac{\pi}{6}\,g_1
     =  \tfrac{\sqrt 3}{2}\,f_1+\tfrac12\,g_1,
   \quad
   u_2  \coloneqq  \cos\tfrac{\pi}{3}\,f_2+\sin\tfrac{\pi}{3}\,g_2
     =  \tfrac12\,f_2+\tfrac{\sqrt 3}{2}\,g_2,
\]
and $U=\Span\{u_1,u_2\}$. Then $p=\dim V=2$, $s=\dim(U\cap V)=0$,
and the principal angles are $(\theta_F,\theta_p)=(\pi/6,\pi/3)$.
The principal-angle frame of Theorem~\ref{lem:frame} is
$\{f_1,f_2\}\subset V$ paired with $\{u_1,u_2\}\subset U$, and
$Mf_k=\sin^2\theta_k\,f_k$ with eigenvalues
$a=\sin^2(\pi/6)=1/4$ and $b=\sin^2(\pi/3)=3/4$.

Let $v^\star\coloneqq\tfrac{\sqrt 3}{2}\,f_1+\tfrac12\,f_2$, the
$\bar v=0$ instance of the worst-case
ray~\eqref{eq:worst-ray} with
$\sqrt b\,f_{s+1}+\sqrt a\,f_p=\sqrt{3/4}\,f_1+\sqrt{1/4}\,f_2$.
Then $\norm{v^\star}=1$, $\bar v^\star=0$, and
\begin{equation}\label{eq:sharp-iterate}
   C_T(v^\star)  =  \rho_V\,(\sqrt b\,f_1-\sqrt a\,f_2)
     =  \tfrac12\bigl(\tfrac{\sqrt 3}{2}\,f_1-\tfrac12\,f_2\bigr).
\end{equation}
In particular $\norm{C_T(v^\star)}=\rho_V=1/2$, so equality
in~\eqref{eq:rate-bound} holds.
\end{example}

\begin{proof}[Verification of~\eqref{eq:sharp-iterate}]
With $w^\star=v^\star=\sqrt{b}\,f_1+\sqrt{a}\,f_2$ and
$Mf_k=\sin^2\theta_k\,f_k$, the computation following
Corollary~\ref{prop:single-step-opt}
gives $\mu_{v^\star}=\scal{w^\star}{Mw^\star}/\scal{w^\star}{M^2w^\star}=2/(a+b)=\mu^\star$,
so by~\eqref{eq:CT-via-M} (with $\bar v=0$),
$C_T(v^\star)=v^\star-\mu^\star Mw^\star=\sqrt b(1-\mu^\star a)f_1+\sqrt a(1-\mu^\star b)f_2$
with $1-\mu^\star a=\rho_V$ and $1-\mu^\star b=-\rho_V$, giving
\eqref{eq:sharp-iterate}.
\end{proof}

\begin{proposition}[Closed-form orbit on the sharp instance]\label{prop:sharp-orbit}
For the instance of Example~\ref{ex:sharp-instance}, the orbit
of $v^\star$ under $C_T$ has the closed form
\begin{equation}\label{eq:sharp-orbit}
   C_T^k(v^\star)  =  \rho_V^k\bigl(\sqrt b\,f_1+(-1)^k\sqrt a\,f_2\bigr),
   \quad
   \norm{C_T^k(v^\star)}=\rho_V^k\quad(k\ge 0),
\end{equation}
where $a=1/4$, $b=3/4$, and $\rho_V=1/2$. In particular,
\eqref{eq:main-rate} holds with equality at every step.
\end{proposition}

\begin{proof}
By induction on $k$. The case $k=0$ is the definition of
$v^\star$. Assume the formula holds at some $k\ge 0$, and write
$w_k\coloneqq C_T^k(v^\star)=\rho_V^k(\sqrt b\,f_1+\epsilon_k\sqrt a\,f_2)$
with $\epsilon_k=(-1)^k\in\{\pm 1\}$. Since $w_k\in V$ and
$\bar w_k=0$ (the orbit lies in $V\cap(U\cap V)^\perp=V$ here
because $s=0$), Theorem~\ref{thm:linesearch}
and~\eqref{eq:CT-via-M} give $C_T(w_k)=w_k-\mu_k Mw_k$ with
$\mu_k=\scal{w_k}{Mw_k}/\scal{w_k}{M^2w_k}$. Using
$Mf_1=af_1$, $Mf_2=bf_2$, and orthonormality of $f_1,f_2$ (with
the squared sign $\epsilon_k^2=1$ collapsing the cross terms),
\[
   \scal{w_k}{Mw_k}=2ab\,\rho_V^{2k},
   \quad
   \scal{w_k}{M^2w_k}=ab(a+b)\,\rho_V^{2k},
\]
so $\mu_k=2/(a+b)=\mu^\star$, independent of $k$. The arithmetic
identities $1-\mu^\star a=\rho_V$ and $1-\mu^\star b=-\rho_V$
then give
\[
   C_T(w_k)  =  \rho_V^{k+1}\bigl(\sqrt b\,f_1-\epsilon_k\sqrt a\,f_2\bigr)
     =  \rho_V^{k+1}\bigl(\sqrt b\,f_1+\epsilon_{k+1}\sqrt a\,f_2\bigr),
\]
using $-\epsilon_k=\epsilon_{k+1}$. This
is~\eqref{eq:sharp-orbit} at $k+1$. The norm identity follows
from $a+b=1$, and equality in~\eqref{eq:main-rate} at every step
is immediate from $\norm{w_{k+1}}/\norm{w_k}=\rho_V$.
\end{proof}

\begin{remark}[Rates at this instance]\label{rem:sharp-rates}
For the parameters $(\theta_F,\theta_p)=(\pi/6,\pi/3)$,
\[
\renewcommand{\arraystretch}{1.25}
\begin{array}{rclrcl}
   c_F &=& \sqrt{3}/2\approx 0.8660, &
   \rho_V &=& 1/2=0.5,\\
   c_F^2 &=& 3/4=0.75, &
   c_F^2/(2-c_F^2) &=& 3/5=0.6,\\
   \rho_{\mathrm{Cheb}} &=& 2-\sqrt 3\approx 0.2679, &
\end{array}
\]
the last by direct substitution into~\eqref{eq:cheb-rate} with
$\sqrt b=\sqrt 3/2$, $\sqrt a=1/2$. The hierarchy
$\rho_{\mathrm{Cheb}}<\rho_V<c_F^2/(2-c_F^2)<c_F^2<c_F<1$ holds
with every inequality strict, in agreement with
Theorem~\ref{thm:cheb-strict}, Propositions~\ref{prop:strict}
and~\ref{prop:worst-case}.
\end{remark}

\subsection{Numerical verification of the sharp rate}\label{ssec:verification}

By Theorem~\ref{thm:rate} the rate $\rho_V$ is the sharp
one-step contraction factor of $C_T$ on $V$: the bound
$\norm{C_T(v)-\bar v}\le\rho_V\,\norm{v-\bar v}$ holds for every
$v\in V$, and equality is attained at the worst-case ray
$v^\star=\bar v+\sin\theta_p\,f_{s+1}+\sin\theta_F\,f_p$. We
verify this identity directly, evaluating the geometric operator
$C_T(x)=\circum\{x,R_Ux,R_VR_Ux\}$ in double-precision arithmetic
and comparing the computed contraction factor with the closed
form of Theorem~\ref{thm:main}.

For prescribed angles $0<\theta_F<\theta_p\le\pi/2$, a pair of
subspaces realizing exactly those principal angles is obtained
in $\RR^4$ by setting
\[
   V=\Span\{e_1,e_2\},\quad
   U=\Span\{\cos\theta_F\,e_1+\sin\theta_F\,e_3,  
            \cos\theta_p\,e_2+\sin\theta_p\,e_4\},
\]
with $e_1,\dots,e_4$ the standard basis. Then $U\cap V=\{0\}$,
the principal vectors of $V$ are $f_1=e_1$ (Friedrichs angle
$\theta_F$) and $f_2=e_2$ (largest principal angle $\theta_p$),
and the worst-case ray reduces to
$v^\star=\sin\theta_p\,e_1+\sin\theta_F\,e_2$.
Table~\ref{tab:verification} reports, for six angle pairs across
the range, the closed-form rate $\rho_V$ of
Theorem~\ref{thm:main} and the contraction factor
$\norm{C_T(v^\star)}/\norm{v^\star}$ computed from the geometric
operator. The two agree to every digit displayed; over the table
the largest discrepancy is $1.1\times10^{-15}$, consistent with
double-precision unit roundoff. The companion rates
$\rho_{\mathrm{Cheb}}$ and $c_F$ are listed alongside, and the
ordering $\rho_{\mathrm{Cheb}}<\rho_V<c_F$ predicted
by~\eqref{eq:rate-hierarchy} and Theorem~\ref{thm:cheb-strict}
is visible in every row.

\begin{table}[htbp]
\centering
\small
\renewcommand{\arraystretch}{1.15}
\begin{tabular}{cc rrc rr}
\toprule
$\theta_F$ & $\theta_p$ &
\multicolumn{1}{c}{$\rho_V$ closed form} &
\multicolumn{1}{c}{$\norm{C_T(v^\star)}/\norm{v^\star}$} &
\multicolumn{1}{c}{$|\Delta|$} &
\multicolumn{1}{c}{$\rho_{\mathrm{Cheb}}$} &
\multicolumn{1}{c}{$c_F$} \\
\midrule
      $\pi/12$ & $\pi/6$ & $0.5773502692$ & $0.5773502692$ & $1.1\times 10^{-16}$ & $0.317837$ & $0.965926$ \\
      $\pi/12$ & $\pi/3$ & $0.8360138566$ & $0.8360138566$ & $0$ & $0.539814$ & $0.965926$ \\
      $\pi/6$ & $\pi/3$ & $0.5000000000$ & $0.5000000000$ & $0$ & $0.267949$ & $0.866025$ \\
      $\pi/6$ & $5\pi/12$ & $0.5773502692$ & $0.5773502692$ & $0$ & $0.317837$ & $0.866025$ \\
      $\pi/4$ & $5\pi/12$ & $0.3021694793$ & $0.3021694793$ & $0$ & $0.154701$ & $0.707107$ \\
      $\pi/6$ & $\pi/2 - 0.01$ & $0.5999679985$ & $0.5999679985$ & $0$ & $0.333311$ & $0.866025$ \\
\bottomrule
\end{tabular}
\caption{Numerical verification of the sharp rate $\rho_V$ of
Theorem~\ref{thm:main}. The closed-form rate (third column) and
the contraction factor of the geometric circumcenter $C_T$ at
the worst-case ray $v^\star$ (fourth column), computed in
double-precision arithmetic, agree to unit roundoff in every
row; the largest discrepancy $|\Delta|$ is $1.1\times 10^{-15}$.
The ordering
$\rho_{\mathrm{Cheb}}<\rho_V<c_F$ predicted
by~\eqref{eq:rate-hierarchy} and Theorem~\ref{thm:cheb-strict}
is visible in every row. The prescribed-angle subspaces
$(U,V)\subset\RR^4$ are constructed as described in
Section~\ref{ssec:verification}.}
\label{tab:verification}
\end{table}

The identity was confirmed beyond these explicit instances on
random subspaces. Over $400$ pairs $U,V$ drawn at random in
$\RR^n$ for $n\in\{20,30,40,60,80\}$ (random number generator
seeded at $2024$), with $\dim U$, $\dim V$, and
$\dim(U\cap V)$ varied, the worst-case ray, obtained from the
eigendecomposition of the self-adjoint operator
$M=(\Id-P_VP_U)|_V$, whose spectrum is $\{\sin^2\theta_k\}$ and
whose extreme eigenvectors furnish $f_{s+1}$ and $f_p$, 
attained the predicted rate $\rho_V$ with maximum error
$8.9\times10^{-16}$. Across $80{,}000$ randomly sampled rays
$v\in V$, none exceeded the bound: the largest recorded value of
$\norm{C_T(v)-\bar v}/\norm{v-\bar v}-\rho_V$ was
$-1.1\times10^{-4}<0$, in agreement with Theorem~\ref{thm:rate}.
These computations use only the construction above and the
definition of $C_T$; the companion code reproduces them from the
stated seed.

\begin{remark}[Finite-precision behavior and circumcenter conditioning]\label{rem:conditioning}
The agreement to unit roundoff reported above is not accidental.
Each evaluation of the geometric circumcenter reduces, by
Definition~\ref{def:circum}, to a single $2\times2$ symmetric
positive-definite system
$G\alpha=\tfrac12(G_{11},G_{22})^{\top}$, with $G$ the Gram
matrix of the edge vectors $R_Ux-x$ and $R_VR_Ux-x$, so the
computed coefficients carry a relative error
$O(\kappa_2(G)\,u)$, $u=2^{-53}\approx 1.1\times10^{-16}$, and
the computed contraction factor inherits it. On the six
instances of Table~\ref{tab:verification},
$\kappa_2(G)\in[1.0,5.3]$ at the worst-case ray, so the
predicted error $\le 6\times 10^{-16}$ accounts for every
observed discrepancy; over the $400$ random pairs,
$\kappa_2(G)$ remained below $11.4$. Neither potentially
degenerate regime harms the computation: as
$\theta_F\to\theta_p$ (the one-step limit $\rho_V\to 0$),
$\kappa_2(G)$ stays bounded, tending to $3$ in the
prescribed-angle family; and as $v\to U\cap V$ the Gram matrix
loses scale rather than conditioning, $\det G\to 0$ like
$\dist^2(v,U\cap V)$ with $\kappa_2(G)$ bounded, so the midpoint
fallback triggers only when the iterate is already at the
solution to working precision.
\end{remark}

\subsection{Iteration counts and the warm-start strategies}\label{ssec:warm-start-numerics}

The closed-form rates established above predict how the
projection methods compare in practice. We close the section
with two experiments, run from the companion code: an
iteration-count comparison of the five methods, and a
measurement of the empirical rate of Strategy~B of
Section~\ref{sec:outside-V}.

For the first experiment we take the prescribed-angle pairs
$(U,V)$ of Section~\ref{ssec:verification}, start every method
from the worst-case ray $v^\star$ of Theorem~\ref{thm:rate}, and
count the iterations needed to bring the relative residual
$\norm{\,\cdot-\bar x_0}/\norm{v^\star-\bar x_0}$ below
$10^{-12}$. The five methods are the Douglas--Rachford operator
$T_{\mathrm{DR}}$, alternating projections $T$, the optimally
relaxed map $S_{\mu^\star}$, the circumcentered--reflection
method $C_T$ on $V$ (Strategy~A), and the Chebyshev
semi-iteration of Section~\ref{ssec:chebyshev}.
Table~\ref{tab:iteration-counts} reports the counts; they track
the rates $c_F$, $c_F^2$, $\rho_V$, $\rho_V$, and
$\rho_{\mathrm{Cheb}}$. The counts for $C_T$ on $V$ and
$S_{\mu^\star}$ coincide, both are below the count for
alternating projections, and Chebyshev is lower still, exactly
the ordering $\rho_{\mathrm{Cheb}}<\rho_V<c_F^2<c_F$
of~\eqref{eq:rate-hierarchy} and Theorem~\ref{thm:cheb-strict}.

\begin{table}[htbp]
\centering
\small
\renewcommand{\arraystretch}{1.15}
\begin{tabular}{cc rrrrr c}
\toprule
$\theta_F$ & $\theta_p$ &
$T_{\mathrm{DR}}$ & MAP & $S_{\mu^\star}$ & $C_T$ on $V$ & Chebyshev &
$\rho_V$ \\
\midrule
      $\pi/12$ & $\pi/6$ & 794 & 397 & 51 & 51 & 25 & $0.577350$ \\
      $\pi/12$ & $\pi/3$ & 796 & 398 & 155 & 155 & 46 & $0.836014$ \\
      $\pi/6$ & $\pi/3$ & 192 & 96 & 40 & 40 & 22 & $0.500000$ \\
      $\pi/6$ & $5\pi/12$ & 192 & 96 & 51 & 51 & 25 & $0.577350$ \\
      $\pi/4$ & $5\pi/12$ & 80 & 40 & 24 & 24 & 16 & $0.302169$ \\
\bottomrule
\end{tabular}
\caption{Iteration counts to bring the relative residual
$\norm{\,\cdot\,-\bar x_0}/\norm{v^\star-\bar x_0}$ below
$10^{-12}$, starting from the worst-case ray $v^\star$ of
Theorem~\ref{thm:rate}, for the five methods of
Section~\ref{ssec:warm-start-numerics}. The counts track the
respective asymptotic rates $c_F$, $c_F^2$, $\rho_V$, $\rho_V$,
and $\rho_{\mathrm{Cheb}}$. The counts for $C_T$ on $V$ and
$S_{\mu^\star}$ coincide
(Theorem~\ref{thm:three-agree-on-V}), and Chebyshev is fastest
in every row, exactly the ordering
$\rho_{\mathrm{Cheb}}<\rho_V<c_F^2<c_F$
of~\eqref{eq:rate-hierarchy} and Theorem~\ref{thm:cheb-strict}.}
\label{tab:iteration-counts}
\end{table}

Figure~\ref{fig:traj-experiment} compares five methods on a
logarithmic scale, adding the parametric
$C_{T,\gamma^\star,\beta^\star}$ tuned to the principal angles
and the Chebyshev semi-iteration to the earlier three. The
slopes at large $k$ match the asymptotic rates
$\rho_{\mathrm{Cheb}}<\Gamma^\star<\rho_V<c_F^2<c_F$ predicted
by Theorems~\ref{thm:main} and~\ref{thm:cheb-strict} and
Section~\ref{ssec:linear-competitor}, and the ordering is visible from the
second iterate onward. Among the five methods, only $C_T$ is
parameter-free, and it attains $\rho_V$ from the first step
without using $\theta_F$ or $\theta_p$.

\begin{figure}[h]
\centering
\includegraphics[width=11cm]{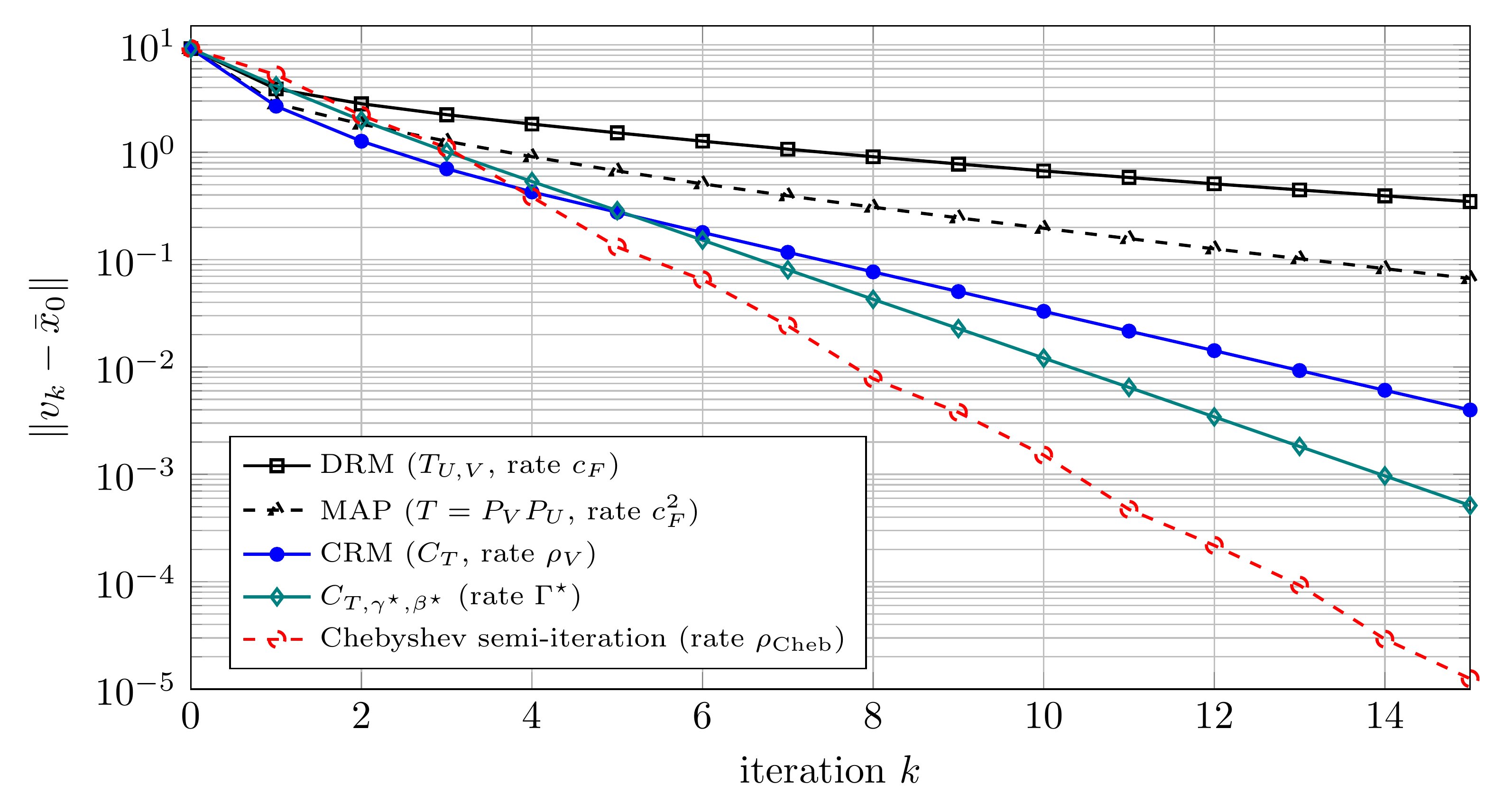}
\caption{Decay of $\norm{v_k-\bar x_0}$ for five methods, started at
$v_0\in V$ with $\norm{v_0-\bar x_0}=9.246$. Asymptotic per-step
rates (visible as the slopes of each curve at large $k$) match the
theoretical predictions: DRM at $c_F\approx 0.899$, MAP at
$c_F^2\approx 0.809$, CRM at $\rho_V\approx 0.679$,
$C_{T,\gamma^\star,\beta^\star}$ at $\Gamma^\star\approx 0.507$, and
Chebyshev at $\rho_{\mathrm{Cheb}}\approx 0.392$. Within the
parameter-free single-step class, $C_T$ dominates; the parametric
$C_{T,\gamma^\star,\beta^\star}$ requires the principal angles to be
known and tuned; Chebyshev requires both angle knowledge and two-step
memory.}
\label{fig:traj-experiment}
\end{figure}

The second experiment produces the Strategy~B statistics quoted
in Remark~\ref{rem:strategy-B-empirical}: the $600$ random
starts are drawn in $\RR^n$ for $n\in\{20,30,40\}$, with the
generator seeded at $20260519$; the proven uniform rate $c_F$ of
Proposition~\ref{prop:strategy-B} is respected by every start
and not improved upon, and the sharp in-$V$ rate $\rho_V$ of
Theorem~\ref{thm:main} is never seen without a warm start.

\subsection{Parameter-free CRM vs.\ parameter-tuned
AAMR}\label{ssec:aamr-comparison}

The rate
$(1-\sin\theta_F)/(1+\sin\theta_F)$ of GAP~\citep{FaltGiselsson2017}
and AAMR~\citep{AragonCampoy2019} depends on the Friedrichs
angle alone; $\rho_V$ depends on both principal angles and
vanishes at $\theta_F=\theta_p$. The two rates therefore cross,
and the location of the crossover in the
$(\theta_F,\theta_p)$ plane is what determines which method
takes fewer iterations on a given problem. We measure this
directly. For every pair $(\theta_F,\theta_p)$ in a triangular
grid $\theta_F\le\theta_p$ of $66$ points in $(0,\pi/2)$, the
companion code builds the prescribed-angle pair $(U,V)$ in
$\RR^4$ realizing those angles, starts both methods from the
worst-case ray $v^\star$ of CRM, and counts iterations to
relative residual $10^{-10}$. The parameter-free $C_T$
contracts at exactly $\rho_V$ per step from this start
(Corollary~\ref{prop:rho-sharp}); AAMR is run at the
optimal parameters $\alpha^\star=1$,
$\beta^\star=1/(1+\sin\theta_F)$
of~\citep[Section~3]{AragonCampoy2019}, and its shadow sequence
$P_U(x_k+v^\star)$ contracts asymptotically at
$(1-\sin\theta_F)/(1+\sin\theta_F)$. The counts are reported
in Table~\ref{tab:aamr-comparison} for the worst-case slice
$\theta_p$ closest to $\pi/2$ in the grid, and the full grid is
displayed as a heatmap in Figure~\ref{fig:aamr-heatmap}.

\begin{figure}[htbp]
\centering
\includegraphics[width=11cm]{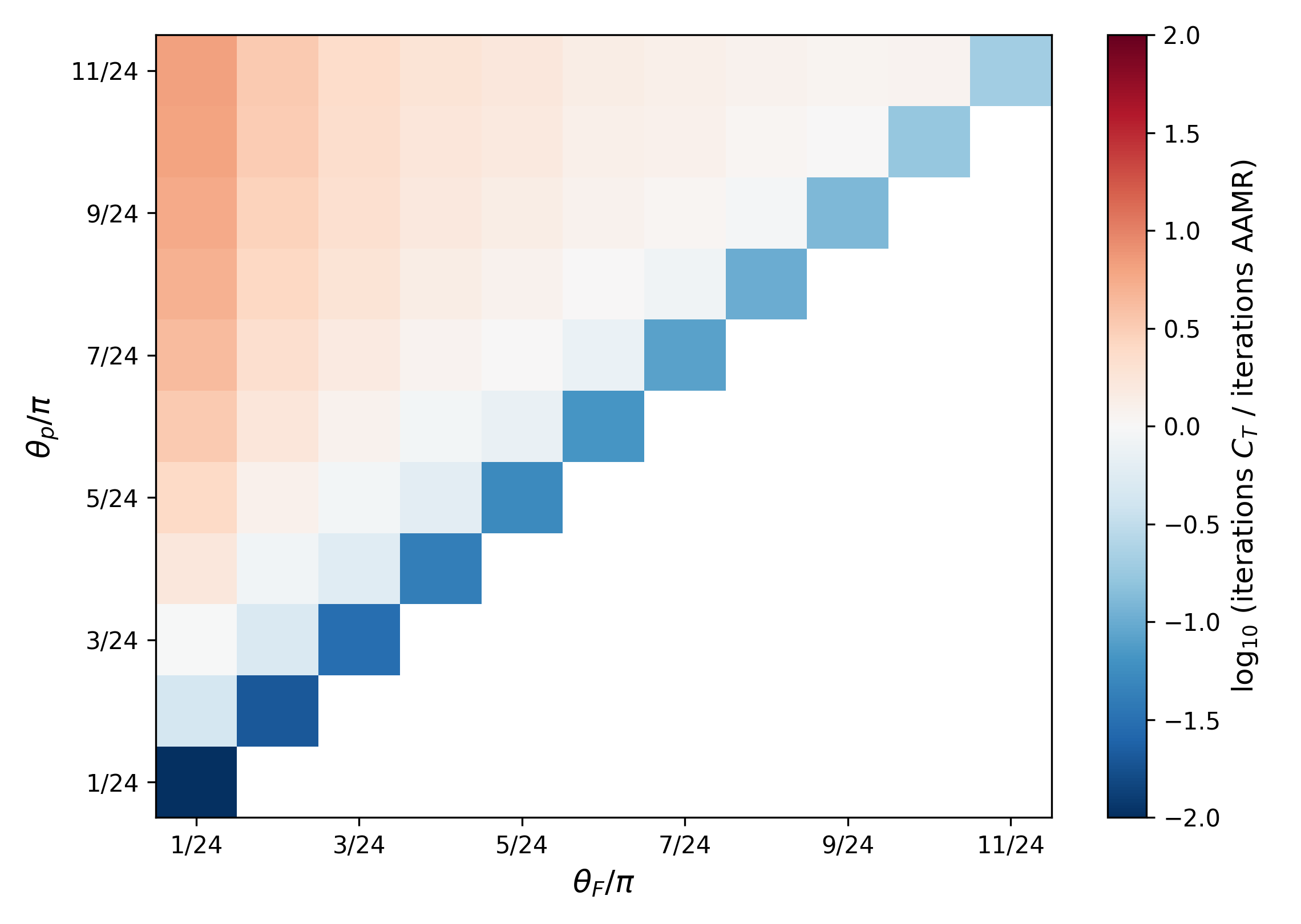}
\caption{Iteration-count ratio
$\log_{10}(\text{iterations of }C_T\text{ on
}V/\text{iterations of AAMR})$ over the full triangular grid
$\theta_F\le\theta_p$,
$\theta_F,\theta_p\in\{k\pi/24\}_{k=1}^{11}$, under the
protocol of this subsection. Blue cells: the parameter-free
$C_T$ wins, most strongly on and near the diagonal
$\theta_F=\theta_p$, where it terminates in one step. Red
cells: parameter-tuned AAMR wins, most strongly at small
$\theta_F$ with $\theta_p$ near $\pi/2$, where
$(1-\sin\theta_F)/(1+\sin\theta_F)\ll\rho_V$. The crossover
between the two regimes matches the rate curves of
Figure~\ref{fig:rates-aamr}.}
\label{fig:aamr-heatmap}
\end{figure}

\begin{table}[htbp]
\centering
\renewcommand{\arraystretch}{1.15}
\begin{tabular}{lcccccc}
\toprule
$\theta_F/\pi$ & $0.0417$ & $0.0833$ & $0.1250$ & $0.1667$ & $0.2083$ & $0.2500$\\
\midrule
$\rho_V$              & $0.9659$ & $0.8724$ & $0.7407$ & $0.5945$ & $0.4524$ & $0.3257$\\
$(1-s_F)/(1+s_F)$     & $0.7691$ & $0.5888$ & $0.4465$ & $0.3333$ & $0.2432$ & $0.1716$\\
iters $C_T$ on $V$    & $665$    & $169$    & $77$     & $45$     & $30$     & $21$   \\
iters AAMR opt.       & $100$    & $50$     & $33$     & $24$     & $19$     & $15$   \\
\midrule
$\theta_F/\pi$        & $0.2917$ & $0.3333$ & $0.3750$ & $0.4167$ & $0.4583$ & \\
\midrule
$\rho_V$              & $0.2193$ & $0.1344$ & $0.0705$ & $0.0261$ & $0.0000$ & \\
$(1-s_F)/(1+s_F)$     & $0.1152$ & $0.0718$ & $0.0396$ & $0.0173$ & $0.0043$ & \\
iters $C_T$ on $V$    & $16$     & $12$     & $9$      & $7$      & $1$      & \\
iters AAMR opt.       & $12$     & $10$     & $8$      & $7$      & $5$      & \\
\bottomrule
\end{tabular}
\caption{Iteration counts to relative residual $10^{-10}$ for
the parameter-free $C_T$ on $V$ versus parameter-tuned AAMR at
the optimum of \citep{AragonCampoy2019}, on the slice
$\theta_p=0.4583\,\pi$ of the experimental grid (close to the
worst case $\theta_p=\pi/2$ for $\rho_V$). At small $\theta_F$
the smaller AAMR rate produces sharply lower counts; at
moderate $\theta_F$ the methods are within a factor of $2$; the
diagonal $\theta_F=\theta_p$, where $C_T$ converges in one step,
is not reached in this slice.}
\label{tab:aamr-comparison}
\end{table}

Two patterns emerge across the full grid (companion code,
block~E). When $\theta_p$ is close to $\theta_F$, $\rho_V$ is
small and $C_T$ on $V$ is markedly faster; on the diagonal
$\theta_F=\theta_p$ it converges in a single step, while AAMR
still takes between $5$ and $100$ iterations depending on
$\theta_F$. When $\theta_p$ is close to $\pi/2$ and $\theta_F$
is small, the gap between $(1-\sin\theta_F)/(1+\sin\theta_F)$
and $\rho_V$ is largest and AAMR is faster, by up to a factor
of $6.6$ in iteration count on the slice of
Table~\ref{tab:aamr-comparison}. Over the $66$ pairs of the
grid, AAMR takes fewer iterations on $39$, $C_T$ on $V$ on
$24$, with $3$ ties: neither method dominates, and which one
wins is determined by the position of $(\theta_F,\theta_p)$
relative to the crossover. This is the trade-off recorded in
Table~\ref{tab:method-comparison}: the AAMR rate depends on
$\theta_F$ alone, buying a uniformly smaller worst-case rate at
the cost of having to know $\theta_F$ to set
$\beta^\star=1/(1+\sin\theta_F)$, while $C_T$ on $V$ uses no
parameter at all, paying for it with a larger rate when
$\theta_p\approx\pi/2$ but gaining substantially when $\theta_p$
is moderate and on the diagonal.
Figure~\ref{fig:rates-aamr} plots the two rates at three values
of $\theta_p$, making the dependence on the principal angles
concrete.

\begin{figure}[h]
\centering
\includegraphics[width=11cm]{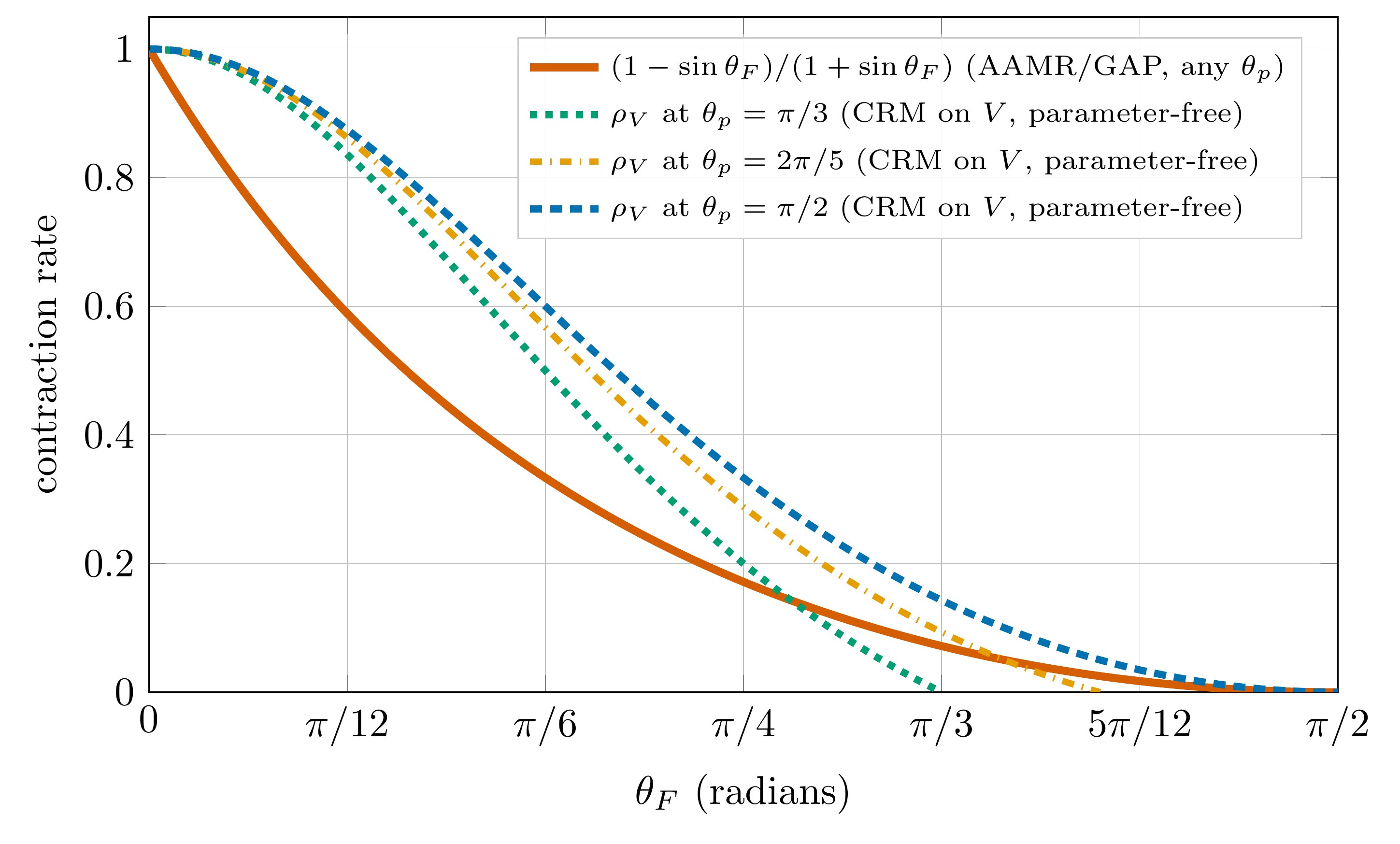}
\caption{Optimal AAMR/GAP rate of Arag\'on
Artacho--Campoy~\citep{AragonCampoy2018,AragonCampoy2019} and
F\"alt--Giselsson~\citep{FaltGiselsson2017} compared with the
parameter-free CRM rate $\rho_V(\theta_F,\theta_p)$ from
Theorem~\ref{thm:main} at three values of $\theta_p$. The AAMR
and GAP rate $(1-\sin\theta_F)/(1+\sin\theta_F)$, attained at
the optimal parameter $\beta^\star=1/(1+\sin\theta_F)$, depends
only on the Friedrichs angle; the parameter-free CRM rate
$\rho_V$ depends on both angles, vanishes at
$\theta_F=\theta_p$ (one-step convergence,
Corollary~\ref{cor:one-step}), and equals
$c_F^2/(2-c_F^2)$ at $\theta_p=\pi/2$
(Proposition~\ref{prop:worst-case}). The two crossover patterns
behind the AAMR$/$CRM iteration counts of
Table~\ref{tab:aamr-comparison} are visible: AAMR is faster
where $\theta_p\approx\pi/2$ and $\theta_F$ is small (the
$\theta_p=\pi/2$ curve of $\rho_V$ sits above the AAMR curve),
while $C_T$ on $V$ is faster on and near the diagonal
$\theta_F=\theta_p$, where $\rho_V$ dips below the AAMR rate
and ultimately vanishes.}
\label{fig:rates-aamr}
\end{figure}

\section{Concluding remarks and open questions}\label{sec:conclusion}

The main result is the sharp rate
$\rho_V=(\sin^2\theta_p-\sin^2\theta_F)/(\sin^2\theta_p+\sin^2\theta_F)$
for $C_T$ started in $V$, attained on an explicit ray and optimal
among all steps along the MAP direction. The constant itself is
not new: it is the optimal rate of the relaxed family
$\{S_\mu\}_{\mu\in\RR}$ and of the linesearch map $B_T$, both due
to Bauschke, Bello-Cruz, Nghia, Phan, and
Wang~\citeyearpar{BBNPW2016}. Nor is the closed form of the step,
which specializes the halfspace representation of CRM steps from
the convex-feasibility literature~\citep{Behling-SIOPT2024}
(Remark~\ref{rem:known-formulas}). What is new is the chain that
connects them. The circumcenter, as the projection of the
solution onto the plane of its vertices, is the error-minimizing
relaxation of MAP; on $V$ it coincides with $A_T$ and $B_T$
(Theorem~\ref{thm:three-agree-on-V}); and the direction of the
step is a self-adjoint operator applied to the error, so that
Kantorovich's inequality delivers the rate, its maximizers, and
the worst-case ray $v^\star=\sin\theta_p\,f_{s+1}+\sin\theta_F\,f_p$
in closed form. The strict inequality $\rho_V<c_F^2$ holds
whenever $\theta_F<\pi/2$, with worst-case envelope
$c_F^2/(2-c_F^2)$, and $C_T$ on $V$ terminates in one step
exactly when $\theta_F=\theta_p$. Chebyshev semi-iteration on $T$
attains the strictly smaller rate
$\rho_{\mathrm{Cheb}}=(\sqrt b-\sqrt a)/(\sqrt b+\sqrt a)$, with
$\rho_V/\rho_{\mathrm{Cheb}}\in(1,2]$
(Theorem~\ref{thm:cheb-strict}), at the price of memory and angle
knowledge. Relaxing the reflections inside the circumcenter
cannot beat $C_T$ at any point of $V$, since the relaxed vertices
span the same plane (Proposition~\ref{prop:over-reflection-double}).
A fixed affine combination of the CRM vertices beats $\rho_V$
only on $\RR^n$, only in the ill-conditioned regime, and only
with both angles known; it is floored at $1/3$ there, and
composed with $P_V$ it recovers exactly $\rho_V$
(Section~\ref{ssec:linear-competitor}).

One direction left open is whether the linesearch framework can
be pushed to a full Krylov-subspace method that uses the entire
spectrum of $M$. Conjugate gradients applied to $M$ on
$V\cap(U\cap V)^\perp$ would, in exact arithmetic, terminate in
at most $\dim V-\dim(U\cap V)$ steps, and would beat
$\rho_{\mathrm{Cheb}}$ when the spectrum of $M$ is
clustered~\citep[Chapter~6]{Saad-IterativeMethods}. Whether the
CG iterates can be produced from $C_T$-style operations, without
an explicit handle on $M$, is unclear.

A few related questions are also open. The exact worst-case
Strategy-B rate on $\RR^n$, in particular whether it equals
the uniform bound $c_F$ of~\eqref{eq:strategy-B}, we leave
open; the experiments of
Section~\ref{ssec:warm-start-numerics} suggest $c_F$ is
essentially the operative rate, and the sharp in-$V$ rate
$\rho_V$ requires the warm start of Strategy~A or~C. Whether a
Chebyshev-type acceleration of $C_T$ itself, exploiting its
nonlinearity, can beat $\rho_{\mathrm{Cheb}}$, we do not know.
An adaptive variant of the linear two-parameter family
$C_{T,\gamma,\beta}$ with $(\gamma_v,\beta_v)$ chosen from local
information at each iterate, parameter-free in the sense of
$C_T$, would be worth exploring.

\section*{Acknowledgements}

I thank the three anonymous referees for their careful reading of
the manuscript. Their reports led to a reorganized introduction, a
shorter and more direct main argument, and a clearer account of the
relation with existing circumcenter formulas; the presentation
improved substantially as a result.

\section*{Statements and Declarations}

\paragraph{Funding.}
This work was supported by the NSF Grant No.~DMS-2307328.

\paragraph{Competing interests.}
The author declares no competing interests.

\paragraph{Data and code availability.}
No external data were generated or analyzed during this study. The
experiments of Section~\ref{sec:numerics} use only the explicit and
randomly generated subspace constructions described there, with all
inputs produced by the companion code. The Julia implementation
behind every numerical result, figure, and table is available at
\url{https://github.com/yunierbello/Sharp-linear-convergence-CRM};
random seeds are fixed, so the reported results reproduce exactly.

\end{document}